\documentclass{article}
\usepackage{amsmath,inputenc}

\usepackage{graphicx}
\usepackage{csvsimple}
\usepackage{rotating}
\usepackage[letterpaper, portrait, margin=1in]{geometry}
\usepackage{hyperref}
\hypersetup{
    colorlinks=true,
    linkcolor=blue,
    citecolor=red,
    filecolor=magenta,      
    urlcolor=cyan
    }

\usepackage[nocomma]{optidef}
\usepackage{listings}
\usepackage{comment}
\usepackage{appendix}
\usepackage[ruled,vlined]{algorithm2e}

\usepackage{subcaption}
\usepackage{booktabs}

\SetAlFnt{\small}
\SetAlCapFnt{\small}
\SetAlCapNameFnt{\small}
\SetAlCapHSkip{0pt}
\IncMargin{-\parindent}

\usepackage{amsfonts} 
\usepackage{amssymb}
\usepackage{bm}
\usepackage[square,comma,numbers]{natbib}

\usepackage{color} 
\definecolor{mygreen}{RGB}{28,172,0} 
\definecolor{mylilas}{RGB}{170,55,241}
\DeclareFixedFont{\ttb}{T1}{txtt}{bx}{n}{12} 
\DeclareFixedFont{\ttm}{T1}{txtt}{m}{n}{12}  
\usepackage{bbm}
\usepackage{amsmath,amsthm}

\newtheorem{theorem}{Theorem}
\newtheorem{corollary}{Corollary}
\newtheorem{proposition}{Proposition}

\newtheorem{lemma}{Lemma}

\newtheorem{problem}{Problem}

\theoremstyle{definition}
\newtheorem{definition}{Definition}
\newtheorem{remark}{Remark}

\definecolor{mygreen}{RGB}{28,172,0} 
\definecolor{mylilas}{RGB}{170,55,241}

\DeclareMathOperator*{\argmin}{arg\,min}

\usepackage[font={small,it}]{caption}

\usepackage[acronym]{glossaries}
\newacronym{otr}{OTR}{Online Traffic Routing}

\newacronym{otr-i}{OTR-I}{Online Traffic Routing with Identical Users}

\usepackage{tikz,pgfplots}
\pgfplotsset{compat=1.14}
\pgfplotsset{scaled y ticks=false}

\definecolor{mycolor1}{rgb}{0.00000,0.44700,0.74100}%
\definecolor{mycolor2}{rgb}{0.85000,0.32500,0.09800}%
\definecolor{mycolor3}{rgb}{0.00000,0.44700,0.74100}%
\definecolor{mycolor4}{rgb}{0.00000,0.44700,0.74100}%
\definecolor{mycolor5}{rgb}{0.85000,0.32500,0.09800}%
\definecolor{mycolor6}{rgb}{0.92900,0.69400,0.12500}%
\definecolor{mycolor7}{rgb}{0.49400,0.18400,0.55600}%
\definecolor{mycolor8}{rgb}{0.46600,0.67400,0.18800}%
\definecolor{mycolor9}{rgb}{0.00000,0.44700,0.74100}%
\definecolor{mycolor10}{rgb}{0.85000,0.32500,0.09800}%
\definecolor{mycolor11}{rgb}{0.00000,0.44700,0.74100}%
\definecolor{mycolor12}{rgb}{0.85000,0.32500,0.09800}%

\newenvironment{hproof}{%
  \renewcommand{\proofname}{Proof (Sketch)}\proof}{\endproof}

\usepackage{color}
\definecolor{deepblue}{rgb}{0,0,0.5}
\definecolor{deepred}{rgb}{0.6,0,0}
\definecolor{deepgreen}{rgb}{0,0.5,0}

\usepackage{listings}

\renewcommand\footnotemark{}

\newif\ifarxiv   
\arxivfalse 

\ifarxiv
\else
\fi

\title{Online Traffic Routing: Deterministic Limits and Data-driven Enhancements}

\author{Devansh Jalota$^1$, Dario Paccagnan$^2$, Maximilian Schiffer$^3$,  and Marco Pavone$^1$%
\thanks{$^1$ Stanford University, USA; {\tt \{djalota, pavone\}@stanford.edu}.}%
\thanks{$^2$ Imperial College London, United Kingdom; {\tt d.paccagnan@imperial.ac.uk}.}%
\thanks{$^3$ Technical University of Munich, Germany; {\tt schiffer@tum.de}.}%
}

\date{}

\begin{document}

\maketitle

\begin{abstract}
    Over the past decade, GPS enabled traffic applications, such as Google Maps and Waze, have become ubiquitous and have had a significant influence on billions of daily commuters' travel patterns. A consequence of the online route suggestions of such applications, e.g., via greedy routing, has often been an increase in traffic congestion since the induced travel patterns may be far from the system optimum. Spurred by the widespread impact of traffic applications on travel patterns, this work studies online traffic routing in the context of capacity-constrained parallel road networks and analyzes this problem from two perspectives. First, we perform a worst-case analysis to identify the limits of deterministic online routing and show that the ratio between the total travel cost of the online solution of any deterministic algorithm and that of the optimal offline solution is unbounded, even in simple settings. This result motivates us to move beyond worst-case analysis. Here, we consider algorithms that exploit knowledge of past problem instances and show how to design data-driven algorithms whose performance can be quantified and formally generalized to unseen future instances. We present numerical experiments based on an application case for the San Francisco Bay Area to evaluate the performance of our approach. Our results show that the data-driven algorithms we develop outperform commonly used greedy online-routing algorithms.
\end{abstract}

\section{Introduction}

The emergence of traffic navigational applications in the past decade has had a transformational impact on billions of people's daily commutes. Particularly, mobile applications that instantaneously suggest routes to user queries have become increasingly popular as they enable free access to real-time traffic information. In the US alone, over 200 million people used Google Maps, Apple Maps, or Waze in 2018 \citep{maps-users-2018}. While the ubiquity of such GPS-enabled routing applications has made travel more convenient, reliable, and accessible, it has also resulted in an ever-growing influence on the traffic patterns of cities, and has often been associated with increasing traffic congestion \citep{Cabannes2017TheIO}. For instance, greedy routing \citep{Hart1968}, wherein route suggestions aim to minimize a user's travel time, often results in a traffic pattern far from the system optimum \citep[cf.][]{TAP,Nash48}. 

In parallel, the advent of big data analytics has led to a new paradigm in traffic routing, namely \textit{anticipatory} route guidance \citep{BILALI2019494,anticipatory-route-guidance}. In this context, the availability of prior traffic information enables the prediction of future demands and traffic conditions, which can then be incorporated in the routing recommendations. In related applications, e.g., online resource allocation, it has been shown that as more prior information on arriving input becomes available the loss due to online decision making decreases \citep{Hwang2018OnlineRA}. 

The impact of traffic navigational applications on travel patterns coupled with the value of prior information in improving the quality of online decisions motivates us to study online traffic routing for parallel road networks from two perspectives. First, we identify the limits of conventional deterministic algorithms, e.g., a greedy algorithm, from a competitiveness analysis lens. We then move beyond worst-case analysis, and introduce data-driven algorithms with probabilistic generalization guarantees. Further, we evaluate the performance of these algorithms through numerical studies. While the focus on parallel road networks may appear restrictive, such networks are of theoretical and practical significance. For instance, Pigou's network \citep{pigou}, a two-arc parallel network, has been used extensively in the selfish routing literature to study price of anarchy bounds \citep{selfish-routing} and more general parallel networks have been used to study Stackelberg routing strategies \citep{amin-stackelberg}. More broadly, analogues of parallel networks have been considered in several online resource allocation problems, including online knapsack \citep{online-knapsack-chakrabarty} and job scheduling~\citep{jobScheduling}. From a practical perspective, transportation networks are rarely parallel but can often be sufficiently approximated by a parallel network for specific use cases \citep{caltrans}, e.g., when multiple highways connect two populated areas, such as the highways connecting San Francisco and San Jose in California.



\subsection{Model} \label{modelling-framework}
We consider a graph with $M \geq 2$ parallel arcs and two vertices representing an origin and a destination. Travelling through arc $a$ requires $t_a$ time units, and each arc $a$ has a capacity $c_a$. We collect these parameters in $\mathcal{S} := \{ \{ c_a \}_{a \in [M]}, \{ t_a \}_{a \in [M]} \}$, where $[N] = \{1, 2, \ldots, N \}$ for any $N \in \mathbb{N}$, and order the arcs, without loss of generality, so that $t_1 \leq t_2 \leq ... \leq t_M$. We denote an input sequence of $n$ users by $\mathcal{I} := \{ \{ \theta_i \}_{i \in [n]}, \{ \tau_i \}_{i \in [n]} \}$, stating for each user $i$ its time of appearance $\tau_i\in\mathbb{R}_{\ge0}$ and its value-of-time $\theta_i \in \Theta$, where $\Theta$ is finite. We assume without loss of generality that $\tau_1 = 0 < \tau_2 < ... < \tau_n $, i.e., no two users arrive in the system at the same point in time. With this notation, the \acrlong{otr} (\acrshort{otr}) problem is as follows.
\begin{problem}\label{prob:otr} (\acrlong{otr} (\acrshort{otr})) Let the parameter set $\mathcal{S}$ and an input sequence $\mathcal{I}$ be given and define a problem instance. Users arrive sequentially in the system and must be irrevocably assigned to an arc at their time of appearance $\tau_i$. The number of users concurrently traversing an arc is limited to the capacity $c_a$ of the arc, and each user traverses an arc $a$ for $t_a$ time units. The \acrshort{otr} problem is to construct an assignment of users to arcs such that the total travel cost, i.e., the sum of the value-of-time weighted travel time of all users, is minimal.
\end{problem} 

If all users are identical, i.e., $\theta_i = \theta_{i'}$ for each $i, i' \in [n]$, we refer to this problem variant as \acrshort{otr-i}. %

We evaluate the efficacy of an algorithm $\mathbf{A}$ that solves Problem~\ref{prob:otr} via its competitive ratio 
$$CR(\mathbf{A}) = \max_{\mathcal{I} \in \Omega} \frac{ALG(\mathcal{I}, \mathbf{A})}{OPT(\mathcal{I})},$$
where $OPT(\mathcal{I})$ denotes the optimal offline travel cost for the \acrshort{otr} problem and $ALG(\mathcal{I}, \mathbf{A})$ denotes the travel cost of the online solution achieved by algorithm $\mathbf{A}$ on input sequence~$\mathcal{I}$. We say that an algorithm is $\alpha$-competitive if $CR(\mathbf{A}) \leq \alpha$. The input sequence $\mathcal{I}$ belongs to a set $\Omega$, which corresponds to user arrivals such that the value-of-time $\theta_i > 0$ and the arrival time $\tau_i \geq 0$ for all users $i$, with the first user arriving at time zero.

%


We compute $OPT(\mathcal{I})$ by solving the following optimization problem
\begin{mini!}|s|[2]<b>
	{\substack{x_{ia}, \\ \forall i \in [n], a \in [M]}}{\sum_{i = 1}^{n} \sum_{a =1}^{M} x_{ia} t_a \theta_i\label{eq:OPT-integer-Obj}}
	{\label{eq:Example3}}
	{OPT(\mathcal{I})=}
	\addConstraint{\sum_{a = 1}^{M} x_{ia}}{=1, \quad \forall i \in [n], \label{eq:prob-constraint}}
	\addConstraint{x_{ia}}{\in \{0, 1\}, \quad \forall i\in [n], a \in [M], \label{eq:binary-constraint}}
	\addConstraint{\sum_{i = 1}^n \mathbbm{1}_{\tau_{k} \in [\tau_{i}, \tau_{i} + t_a]} x_{ia}}{\leq c_a, \quad \forall a \in [M], \forall \tau_k \in \{\tau_1, ..., \tau_n \}, \label{eq:capacity-constraints}}
\end{mini!}
where we introduced binary variables $x_{ia}$ to denote whether user $i$ is assigned to arc $a$ ($x_{ia} = 1$) or not ($x_{ia}=0$). Here, \eqref{eq:prob-constraint} are assignment constraints,~\eqref{eq:binary-constraint} are binary constraints and~\eqref{eq:capacity-constraints} are capacity constraints. Throughout the manuscript, we assume that Problem~\eqref{eq:OPT-integer-Obj}-\eqref{eq:capacity-constraints} is feasible, i.e., the capacity constraints can be satisfied.

We conclude this section highlighting that our model mirrors the operation of commonly used traffic navigational applications such as Google Maps or Waze that suggest routes to users. Furthermore, our model can be used to describe the online operation of a centrally controlled smart mobility system, e.g., a fleet of autonomous taxis. Here, the traffic application or the fleet operator takes centralized routing decisions and it is assumed that users are not strategic, i.e., the users' values of time are known by the traffic navigational application or the fleet operator. While, in general, user's values of time may not be known, it can often be estimated~\citep{vot-learn}. Accordingly, most literature on traffic routing with heterogeneous users~\citep{heterogeneous-pricing-roughgarden,multicommodity-extension} make this assumption. Furthermore, in the context of heterogeneous user traffic routing, minimizing the weighted travel time of all users is also standard in the literature~\citep{METRP}.

\subsection{Contribution}
In this work, we develop algorithms for the \acrshort{otr} problem in two settings. We first focus on an adversarial setting and perform a worst-case competitive analysis of deterministic algorithms. We then focus on a stochastic user arrival setting and design a data-driven algorithm whose performance can be provably generalized to unseen instances.

In the adversarial setting, we first consider the \acrshort{otr-i} problem for a two-arc parallel network and show that a greedy algorithm, which allocates users to the least cost arc that is below capacity, is optimal, i.e., achieves a competitive ratio of one. We then consider other problem variants within \acrshort{otr}, including the settings of $M \geq 3$ arcs and non-identical users, and show that the ratio of the total travel cost of the online solution of any deterministic algorithm to that of the optimal offline solution is unbounded. We extend this worst-case analysis to general (non-parallel) road networks and general arc cost functions.

The unbounded competitive ratios encountered for the variants of the two-arc \acrshort{otr-i} problem motivate us to leverage the ever-increasing availability of traffic data, e.g., origin-destination travel information, to devise improved data-driven routing algorithms. In particular, given the availability of such prior travel information, we study the setting of stochastic arrivals, wherein users arrive according to some unknown probability distribution. In this setting, we develop both time-independent and time-dependent data-driven algorithms that learn an optimal parametrization for traffic routing from past training input sequences and utilize this learned information to make online decisions on unseen test sequences. In this context, we obtain probabilistic generalization guarantees for our data-driven algorithms to unseen input sequences drawn from the same probability distribution. 

Finally, we present numerical experiments to evaluate the performance of the data-driven and greedy algorithms for a range of problem settings, including an application case to the San Francisco Bay Area. The results indicate that both the data-driven algorithms achieve lower ratios between the travel costs of the online and optimal offline solutions as compared to that of the greedy algorithm. That is, both the time-independent and time-dependent data-driven algorithms outperform the greedy algorithm, thus highlighting the value of prior information in online algorithm design. Furthermore, the time-dependent algorithm achieves a better performance compared to its time-independent counterpart, although the latter has a superior probabilistic generalization guarantee. This observation points towards the importance of adjusting allocations based on changes in user arrival rates.

\subsection{Organization:} The remainder of this paper is as follows. Section~\ref{lit-review} reviews related literature. In Section~\ref{worst-case}, we study the limits of deterministic algorithms for variants of the \acrshort{otr-i} problem from a competitive-analysis perspective when user arrivals are adversarial. We then move to the setting wherein prior stochastic information on user arrivals is known and develop data-driven algorithms with generalization guarantees (Section~\ref{beyond-worst-case}). Finally, we evaluate the performance of the data-driven algorithms as compared to the greedy algorithm through numerical experiments in Section~\ref{main-numerical} and conclude the paper in Section~\ref{conclusions}.

\section{Literature Review} \label{lit-review}

This paper is closely related to online resource allocation problems and to beyond worst-case online algorithm design. We begin with a review of two dynamic resource allocation applications, namely the online knapsack and job scheduling problems, that are closely related to the model studied in this paper. We then survey results on competitive analysis of online algorithms and beyond worst-case methodologies, and discuss how they relate to this work.

In online knapsack problems, the goal is to maximize the total value of items irrevocably assigned to knapsacks with fixed capacities, without knowledge of the values of items that will arrive in the future \citep{online-knapsack-wierman,online-knapsack-chakrabarty}. While similar to the model we consider, when an item is assigned to a knapsack, the item stays in the knapsack indefinitely, i.e., until no item remains or until no remaining item can fit in any knapsack. In this regard, the online knapsack setting differs from the \acrshort{otr} problem since users allocated to arcs (or equivalently knapsacks) also depart from the arcs when they complete their trip.

In job scheduling problems, $n$ jobs need to be scheduled on $m$ machines to optimize a system objective, such as the total time spent by all jobs in the system. In the online setting, jobs are either assumed to arrive over time (online-time job scheduling), or they may be scheduled in a sequence (online-list job scheduling), cf. \cite{Pruhs03onlinescheduling}. Online-time job scheduling parallels the \acrshort{otr} problem in the user (or job) arrival process since users arrive at different points in time into the network. On the other hand, the allocation of users to arcs instantaneously upon arrival mirrors online-list job scheduling, wherein each job must be scheduled on a machine at the instant that it arrives. 
Thus, the online traffic routing setting we study combines different elements of the two commonly used online job scheduling models. 

Online resource allocation algorithms have traditionally been designed to perform well under worst-case assumptions, i.e., under the assumption that inputs are adversarially chosen. For instance, \citet{online-TSP-jaillet} develop algorithms for the online traveling salesman problem (TSP) with worst-case performance guarantees, while \citet{msvv} and \citet{primal-dual-buchbinder} obtain worst-case competitive ratios for online bipartite matching in the context of internet advertising. Since worst-case competitive analysis can be exceedingly pessimistic, \citet{Hwang2018OnlineRA} demonstrated the value of prior information in online decision making. More specifically, the authors show that the loss due to online decision making decreases with an increase in the proportion of available prior information on arriving input. In line with these works, we first study the \acrshort{otr} problem from a worst-case perspective. To improve the resulting worst-case guarantees, we additionally analyze a setting wherein prior information on the arriving input is known, which is a common informational assumption used in many online resource allocation applications~\citep{devanur-online-resource,motwani-scheduling}.

One approach used to perform beyond worst-case analysis \citep{beyond-comp-analysis,roughgarden2018worstcase} is that of stochastic arrivals, wherein the input arrives according to some potentially unknown probability distribution. Such stochastic assumptions on future input have been used to aid online algorithm design to improve on worst-case competitive ratios in many applications, including online TSP \citep{online-tsp-fair} and online bipartite matching \citep{TCS-057}. In particular, one approach has been to use the offline optimum for an expected instance to make online decisions \citep{stochastic-matching-beating}. In this work, we move beyond such techniques through data-driven algorithms that are trained on past input sequences rather than just on an expected sequence, thereby capturing the stochastic variations in the arriving input.

Another beyond worst-case model is the random order model, wherein the sequence of arrivals occurs according to a random permutation of an arbitrary input sequence. In this model, a small fraction of the entire input sequence is observable, which allows for predicting the pattern of future demand. For instance, \citet{devanur-adwords} use a fraction of the input sequence to train their online bipartite matching model and then use the learned parameters to make online decisions for the remaining unseen input. This approach was further extended to a repeated learning setting in the context of revenue management by \citet{online-agrawal}. Our data-driven approach mirrors the two-phase approach of model training and online decision making adopted in these works. However, there is a fundamental difference between our approach and that of \citet{devanur-adwords} and \citet{online-agrawal}. In particular, we use parameters learned from past input sequences to make prescriptive online decisions on unseen input sequences. In contrast, \citet{devanur-adwords} and \citet{online-agrawal} learn model parameters by observing only a fraction of a given input sequence to make decisions for the remaining input sequence. When making online decisions in the transportation context, we need to allocate users instantaneously. Thus, observing the first $\epsilon$ fraction of arriving users may not be viable. Instead, we can leverage past travel data, e.g., origin-destination and traffic flow data, to learn information on user arrivals and subsequently make online decisions.

This paper contributes to the aforementioned literature in several ways. With respect to online resource allocation literature, we propose a novel model for dynamic resource allocation motivated through online traffic routing applications. With regards to beyond worst-case online algorithm design, we propose data-driven algorithms based on the scenario approach~\citep[cf.][]{scenario-main}, which we believe is of independent interest and is broadly applicable to online resource allocation problems beyond the one studied in this paper. 



\section{Worst-Case Analysis of Online Traffic Routing Algorithms} \label{worst-case}

In this section, we consider the setting of adversarial user arrivals and study worst-case competitive ratios of deterministic algorithms for variants of the \acrshort{otr} problem. To this end, we first consider the \acrshort{otr-i} problem, wherein all users have identical values of time, and show that a greedy algorithm achieves a competitive ratio of one for a two arc parallel network (Section~\ref{identical-customers}). We then show that the ratio of the travel cost of any deterministic online algorithm's solution to that of the optimal offline solution can be arbitrarily large for problem settings beyond two arc parallel road networks with identical users (Section~\ref{Non-Identical-Customers}).

\subsection{Routing Identical Users in a Two-Arc Capacity Constrained Parallel Network} \label{identical-customers}

This section focuses on routing identical users in a two-arc parallel road network and presents a greedy algorithm that achieves a competitive ratio of one. In particular, we analyze the greedy algorithm, described in Algorithm~\ref{alg:greedy}, which allocates arriving users to the least cost arc that is below capacity. Since users are identical, we normalize $\theta_i$ to one for all agents $i$ in this section.
\begin{algorithm} 
\SetAlgoLined
\SetKwInOut{Input}{Input}\SetKwInOut{Output}{Output}
\textbf{When User $i \in [n]$ arrives:} \\
\quad \quad Allocate $i$ to the lowest cost arc that is below capacity
\caption{Greedy Algorithm}
\label{alg:greedy}
\end{algorithm}

We now present the main result of this section, which establishes that the greedy algorithm, described in Algorithm~\ref{alg:greedy}, is optimal for the two-arc \acrshort{otr-i} problem.

\begin{theorem} \label{thm:greedy-main-result} (Optimality of Greedy for \acrshort{otr-i})
Let $\mathbf{A}$ denote Algorithm~\ref{alg:greedy}. Then, $CR(\mathbf{A}) = 1$ for the \acrshort{otr-i} problem when the number of arcs $M$ in the network is two. 
\end{theorem}
We prove Theorem~\ref{thm:greedy-main-result} using two intermediary results. We first characterize a sufficient condition for an algorithm to achieve a competitive ratio of one for a general $M$-arc network and then show that the greedy algorithm satisfies this sufficient condition when $M=2$. To characterize the sufficient condition, we focus on the set of \textit{feasible} algorithms that satisfy Constraints~\eqref{eq:prob-constraint}-\eqref{eq:capacity-constraints}. We note that the set of \textit{feasible} algorithms is non-empty as Algorithm~\ref{alg:greedy} allocates users upon arrival to an arc below capacity and thus satisfies Constraints~\eqref{eq:prob-constraint}-\eqref{eq:capacity-constraints}. Let $\mathbf{A}$ be one such feasible algorithm and let $N_{a, \mathbf{A}}^{\mathcal{I}}$ denote the total number of users that are allocated by $\mathbf{A}$ to arc $a$ for input sequence~$\mathcal{I}$. Then, Lemma~\ref{lem:lemm1} provides a sufficient condition for $\mathbf{A}$ to achieve a competitive ratio of one.
\begin{lemma} \label{lem:lemm1} (Optimality Condition for \acrshort{otr-i})
 Suppose that there is a feasible algorithm $\mathbf{A}$, satisfying Constraints~\eqref{eq:prob-constraint}-\eqref{eq:capacity-constraints} for the \acrshort{otr-i} problem, such that for all input sequences~$\mathcal{I}$, $\sum_{a = 1}^{k} N_{a, \mathbf{A}}^{\mathcal{I}} \geq \sum_{a = 1}^{k} N_{a, \mathbf{A}'}^{\mathcal{I}}$ for $k \in [M]$ holds for all feasible algorithms $\mathbf{A'}$ satisfying Constraints~\eqref{eq:prob-constraint}-\eqref{eq:capacity-constraints}. Then $CR(\mathbf{A}) = 1$ for the \acrshort{otr-i} problem.
\end{lemma}
The above lemma states that an algorithm is optimal for the \acrshort{otr-i} problem if it allocates more users on the lower cost arcs than any other \textit{feasible} algorithm. We prove Lemma~\ref{lem:lemm1} in \ifarxiv Appendix~\ref{proof-lem-1} \else Appendix A.1 \fi. 

Our second intermediary result establishes that Algorithm~\ref{alg:greedy} satisfies the sufficient optimality condition of Lemma~\ref{lem:lemm1} for a two-arc parallel network. To establish this claim, we will show that no \textit{feasible} algorithm can allocate more users to arc one than the greedy algorithm. 

\begin{lemma} \label{lem:greed-opt} (Sufficiency for Optimality of Greedy for Two-Arc \acrshort{otr-i})
In a two arc parallel road network with identical users, Algorithm~\ref{alg:greedy}, denoted as $\mathbf{A}$, satisfies the property that $ N_{1, \mathbf{A}}^{\mathcal{I}} \geq N_{1, \mathbf{A}'}^{\mathcal{I}}$ for all input sequences $\mathcal{I}$ and for any algorithm $\mathbf{A}'$ that satisfies the feasibility Constraints~\eqref{eq:prob-constraint}-\eqref{eq:capacity-constraints}.
\end{lemma}

\begin{hproof}
Denoting arcs one and two by $a_1$ and $a_2$, we proceed by induction on the number of users that Algorithm~\ref{alg:greedy} allocates on $a_2$. For the base case, it is easy to see from the feasibility constraints that if Algorithm~\ref{alg:greedy} allocates one user on $a_2$ then any other \textit{feasible} algorithm must also allocate at least one user on $a_2$ as well. We then consider the inductive step and show that if $k+1$ users are allocated by Algorithm~\ref{alg:greedy} on $a_2$ then at least $k+1$ users are allocated by any other \textit{feasible} algorithm on $a_2$. This follows since the greedy algorithm always allocates users on $a_1$ if it is below capacity. Thus, for every user the greedy algorithm allocates on $a_2$, any other \textit{feasible} algorithm must also allocate at least one user on $a_2$. 
\end{hproof}
\noindent For a detailed proof of Lemma~\ref{lem:greed-opt}, see \ifarxiv Appendix~\ref{proof-greed-opt} \else Appendix A.2 \fi. Finally, noting that algorithms $\mathbf{A}$ and $\mathbf{A}'$ must allocate the same number of users, it follows with equality that $\sum_{a = 1}^{2} N_{a, \mathbf{A}}^{\mathcal{I}} \geq \sum_{a = 1}^{2} N_{a, \mathbf{A}'}^{\mathcal{I}}$. Thus, Algorithm~\ref{alg:greedy} satisfies the optimality condition of Lemma~\ref{lem:lemm1}, thereby establishing Theorem~\ref{thm:greedy-main-result}.

\subsection{Worst-Case Performance Limitations of Deterministic Algorithms} \label{Non-Identical-Customers}
The optimality of the greedy algorithm for the two-arc \acrshort{otr-i} problem raises the question of whether a similar result holds for more general problem variants. In this section, we first consider \acrshort{otr} problem variants beyond the two-arc \acrshort{otr-i} problem by considering the settings of (i) $M \geq 3$ arcs, and (ii) non-identical users. We then consider the extensions of the \acrshort{otr} problem to the settings of (iii) general road networks and (iv) general cost functions for the arcs. We show that for these problem variants  the worst-case ratio between the online and optimal offline solutions is unbounded for any deterministic algorithm.

\subsubsection{Worst-Case Ratios for the \acrshort{otr} Problem.}
In this section, we show that for any variant of the two-arc \acrshort{otr-i} problem, an adversary can construct instances such that the ratio of the total travel cost of the online solution of any deterministic algorithm to that of the optimal offline solution is unbounded. 
\begin{theorem} \label{thm:only-otr-i-optimal} (Unbounded Worst-Case Ratios for \acrshort{otr})
For any \textit{feasible} deterministic algorithm, there are instances such that the ratio of the total travel cost of the online solution to that of the optimal offline solution is unbounded in each of the following cases:
(i) The number of arcs $M$ in the network is at least three, and
(ii) the users are non-identical.
Only in the case of the two-arc \acrshort{otr-i} problem, i.e., all users have equal values of time, there exists a \textit{feasible} deterministic algorithm that achieves a finite ratio between the total travel cost of an online solution and that of the optimal offline solution.
\end{theorem}

This theorem highlights that the assumptions of identical users and a two-arc parallel network are critical to the optimality of Algorithm~\ref{alg:greedy}. In fact, this result implies that the two-arc \acrshort{otr-i} problem is the only variant of \acrshort{otr} for which any deterministic algorithm can achieve a finite ratio between the total travel cost of an online solution and that of the optimal offline solution.

To prove Theorem~\ref{thm:only-otr-i-optimal}, it suffices to focus on two independent problems: (i) $M \geq 3$ arc \acrshort{otr-i}, and (ii) two-arc non-identical user \acrshort{otr}.


\paragraph{$M \geq 3$ Arcs \acrshort{otr-i}:} We first consider the extension of the two-arc \acrshort{otr-i} problem to an $M$ arc parallel network setting where $M \geq 3$. Lemma~\ref{lem:M-greater-than-3} establishes that the worst-case ratio of the total travel costs of the online solution of any deterministic algorithm to that of the optimal offline solution is unbounded in this setting.
\begin{lemma} \label{lem:M-greater-than-3} (Unboundedness for $M \geq 3$ Arcs)
For any $M \geq 3$, there exist arc costs $t_1, t_2, t_3$ and arc capacities $c_1$, $c_2$, and $c_3$ such that the ratio of the total travel cost of the online solution of any deterministic algorithm to that of the optimal offline solution is at least $1 + \frac{t_3-t_2}{2 t_1 + 2 t_2}$. The quantity $1 + \frac{t_3-t_2}{2 t_1 + 2 t_2}$ is unbounded since an adversary can pick $t_3$ to be arbitrarily large.
\end{lemma}

\begin{hproof}
We fix a three arc parallel network, where the arc capacities are $c_1 = 1$, $c_2 = 1$, $c_3 = 10$ and the arc costs are $t_1 = 5$, $t_2 = 10.01$, and $t_3$ can be chosen by an adversary. In any such three-arc network, a \textit{feasible} deterministic algorithm $\mathbf{A}$ must allocate the first arriving user on one of the three arcs. We analyze each case and show that an adversary can construct input sequences such that any \textit{feasible} deterministic algorithm will allocate users on arc three, while the optimal allocation assigns users to the first two arcs. The analysis yields a lower bound on the ratio of the total travel cost of the online solution of any deterministic algorithm to that of the optimal offline solution of $1 + \frac{t_3-t_2}{2 t_1 + 2 t_2}$. Unboundedness follows since $t_3$ can be chosen to be arbitrarily large for fixed $t_1$ and $t_2$.
\end{hproof}
We refer to \ifarxiv Appendix~\ref{proof-M-greater-than-3} \else Appendix A.3 \fi for a detailed proof of Lemma~\ref{lem:M-greater-than-3}. 

\paragraph{Two-arc Non-Identical User \acrshort{otr}:} We now consider the case of non-identical users, i.e., the \acrshort{otr} problem described in Section~\ref{modelling-framework}, and show that the ratio of the total travel cost of the online solution of any \textit{feasible} deterministic algorithm to that of the optimal offline solution is not finite. We further show for the two-arc \acrshort{otr} problem that Algorithm~\ref{alg:greedy} achieves the best possible ratio for the total travel cost of the online solution of any deterministic algorithm to that of the optimal offline solution. 
\begin{lemma} \label{lem:non-identical-OTR} (Unboundedness for Non-Identical Users)
Consider a parallel network where $M=2$ and let $\theta'$ and $\theta''$ be the lowest and highest user values-of-time respectively. Then for the \acrshort{otr} problem there exist arc capacities $c_1$, $c_2$ such that the ratio of the total travel cost of an online solution of any \textit{feasible} deterministic algorithm to that of the optimal offline solution is at least $\frac{\theta'' t_2 + \theta' t_1}{\theta'' t_1 + \theta' t_2}$, where $t_1$ and $t_2$ are the costs of the two arcs such that $t_1 \leq  t_2$. Algorithm~\ref{alg:greedy} achieves this ratio for any two-arc parallel network. 
\end{lemma}

\begin{hproof}
To prove this theorem, we first define the set of network parameters $\mathcal{S}$ and provide two input sequences for which the ratio of the solution of any deterministic algorithm to the optimal offline solution is at least $\frac{\theta'' t_2 + \theta' t_1}{\theta'' t_1 + \theta' t_2}$ on one of these input sequences. Next, to establish that the greedy algorithm achieves this ratio for a two-arc parallel network, we use the property of the greedy algorithm established in Lemma~\ref{lem:greed-opt} and provide a sequence of inequalities, which establish that $\frac{\theta'' t_2 + \theta' t_1}{\theta'' t_1 + \theta' t_2}$ is an upper bound on the ratio of the total travel cost of the online solution of Algorithm~\ref{alg:greedy} to that of the optimal offline solution. Since Algorithm~\ref{alg:greedy} is deterministic, the greedy algorithm achieves the ratio of $\frac{\theta'' t_2 + \theta' t_1}{\theta'' t_1 + \theta' t_2}$ for the \acrshort{otr} problem for a two-arc parallel network.
\end{hproof}

We refer to \ifarxiv Appendix~\ref{proof-non-identical-OTR} \else Appendix A.4 \fi for a detailed proof of Lemma~\ref{lem:non-identical-OTR}. Note that if we let $\theta'' >> \theta'$, then $\frac{\theta'' t_2 + \theta' t_1}{\theta'' t_1 + \theta' t_2} \approx \frac{t_2}{t_1}$. That is, the ratio of the online and offline solutions is at least the ratio of the travel times of the links. Since an adversary can choose $t_2$ to be arbitrarily large the ratio $\frac{t_2}{t_1}$ is unbounded. Furthermore, note that if the lowest and highest values of time are equal, i.e., $\theta' = \theta''$, then the competitive ratio of the greedy algorithm obtained in Lemma~\ref{lem:non-identical-OTR} equals one, which parallels the result of Theorem~\ref{thm:greedy-main-result}. As the difference in the values of time increases the worst case ratio between the online greedy solution and that of the optimal offline solution also increases. Thus, Lemma~\ref{lem:non-identical-OTR} provides a parametric representation of the competitive ratio of the greedy algorithm as we move smoothly between \acrshort{otr-i} to \acrshort{otr} for a two-arc parallel network.

Jointly Lemmas~\ref{lem:M-greater-than-3} and~\ref{lem:non-identical-OTR} establish Theorem~\ref{thm:only-otr-i-optimal}.

\subsubsection{Extensions of the \acrshort{otr} Problem.} In this section, we consider extended variants of the \acrshort{otr} problem to the settings of (i) general non-parallel road networks, and (ii) general cost functions. In both settings, we show that even when users are identical, the ratio between the total travel cost of the online solution of any deterministic algorithm to that of the optimal offline solution is unbounded.

\paragraph{General (Non-Parallel) Road Networks:} We first consider the extension of the \acrshort{otr} problem to general road networks with identical users and refer to \ifarxiv Appendix~\ref{framework-general-road} \else Appendix B.1 \fi for an introduction to this problem setting. Lemma~\ref{lem:general-road-networks} shows that in this setting, the ratio between the total travel cost of the online solution of any deterministic algorithm and that of the the optimal offline solution is again unbounded. 
\begin{lemma} \label{lem:general-road-networks} (Unboundedness for General Road Networks)
Consider a series parallel road network with two sets of parallel networks, each consisting of two arcs. Then for some arc capacities and costs, the ratio between the total travel cost of the online solution of any \textit{feasible} deterministic algorithm to that of the optimal offline solution is unbounded for the generalization of the \acrshort{otr-i} problem to general road networks.
\end{lemma}

\begin{hproof}
To prove the lemma, we consider a series parallel network with one origin and one destination, separated by one intermediate node. We then consider the possible cases for how a deterministic algorithm could allocate the first arriving user in the series parallel network. In each case, we demonstrate an input sequence for which a deterministic algorithm will allocate at least one user on an arc that the optimal offline allocation will never allocate users to. Since an adversary can choose the cost of this arc to be arbitrarily large, the ratio between the total travel cost of the online solution of any deterministic algorithm and that of the optimal offline solution cannot be bounded.
\end{hproof}
We refer to \ifarxiv Appendix~\ref{proof-general-road-networks} \else Appendix A.5 \fi for a detailed proof of Lemma~\ref{lem:general-road-networks}. 

\paragraph{General Cost Functions:} Thus far, we have considered the setting where the cost of the arcs are constant. We now consider a setting where the travel time (rather than being fixed) on the arcs is modelled by a cost function $f_a: \mathbb{R}_{\geq 0} \mapsto \mathbb{R}_{\geq 0}$. Here we denote the travel time of arc $a$ as $f_a(y)$ when $y$ users are on arc $a$. We introduce the relevant notation and present the \acrshort{otr} problem for general cost functions in \ifarxiv Appendix~\ref{framework-general-congestion} \else Appendix B.2 \fi. In this setting, we show that the ratio between the total travel cost of any deterministic algorithm's online solution and that of the optimal offline solution is unbounded.

\begin{lemma} \label{lem:general-congestion-functions} (Unboundedness for General Cost Functions)
Let $M=2$, then for some arc capacities $c_1$ and $c_2$, there exist arc cost functions $f_1:\mathbb{R}_{\geq 0} \mapsto \mathbb{R}_{\geq 0}$ and $f_2:\mathbb{R}_{\geq 0} \mapsto \mathbb{R}_{\geq 0}$ such that the ratio of the total travel cost of the online solution of any \textit{feasible} deterministic algorithm to that of the optimal offline solution is unbounded for the \acrshort{otr-i} problem for general cost functions.
\end{lemma}

\begin{hproof}
We consider a two arc network with arc capacities $c_1 = 2, c_2 = 1$. We further define the cost function of arc two ($a_2$) to be constant and let the cost of arc one ($a_1$) be defined by a linear function $f_1$ that satisfies $f_1(1) = 5$ and $f_1(2) = \Bar{d}$, where $\Bar{d}$ is chosen by an adversary. We then construct input sequences such that any \textit{feasible} deterministic algorithm will allocate two users simultaneously on arc $a_1$ to incur the cost $\Bar{d}$, while the optimal allocation is to assign users such that no two users concurrently use $a_1$. Since $\Bar{d}$ can be made arbitrarily large, the ratio of the total travel cost of the online solution of any deterministic algorithm to that of the optimal offline solution is unbounded.
\end{hproof}
We refer to \ifarxiv Appendix~\ref{proof-general-congestion-functions} \else Appendix A.6 \fi for a detailed proof of Lemma~\ref{lem:general-congestion-functions}.

In this section, we analyzed four problem variants that removed the simplifying assumptions of the two-arc \acrshort{otr-i} problem considered in Section~\ref{identical-customers}. In each of these variants, we established that the ratio between the total travel cost of the online solution of any deterministic algorithm to that of the optimal offline solution is unbounded. 

\section{Beyond Worst-Case Analysis With Prior Information} \label{beyond-worst-case}
The worst-case analysis for the \acrshort{otr} problem motivate us to consider the setting wherein prior information of the user arrivals is known. Assumptions on user arrivals are reasonable in practical traffic routing applications due to the ever-increasing availability of traffic data. For instance, origin-destination travel information enables the capture of regularities in user arrivals, e.g., the similarity in peoples' weekday travel patterns throughout the year. Given the availability of prior information, we consider stochastic assumptions on user arrivals, where the input sequences $\mathcal{I}$ are drawn from a (potentially unknown) probability distribution. For this setting, we devise data-driven algorithms that learn an optimal parametrization for traffic routing through past observed input sequences and utilize this learned information to make prescriptive real-time decisions on unseen test sequences. We further obtain probabilistic generalization guarantees of our data-driven algorithms to unseen input sequences drawn from the same probability distribution.

We begin this section by presenting an overview of the scenario optimization paradigm in Section~\ref{classical-scenario-framework}, which we leverage to devise a time-independent data-driven algorithm for the \acrshort{otr} problem in Sections~\ref{data-driven-oad} and~\ref{offline-parameter-learning}. We finally present generalization guarantees for our time-independent algorithm and develop an extension of this algorithm to the time-dependent setting in Section~\ref{data-driven-online-alg}.

\subsection{Scenario Approach} \label{classical-scenario-framework}

The scenario optimization framework provides a methodology to perform optimization in uncertain environments where the instances are represented by an uncertainty vector $\delta$ that lies in the probability space $(\Delta, \mathcal{F}, \mathbb{P})$. Here $\Delta$ is the sample or uncertainty set of possible outcomes, $\mathcal{F}$ is a $\sigma$-algebra, and $\mathbb{P}$ is a probability measure on the sets of possible outcomes. The scenario approach is based on the following prototypical convex optimization problem:
\begin{align} \label{eq:classic-scenario}
    \min_{\pi \in \Pi} c^{T} \pi \quad \text { subject to }: f\left(\pi, \delta_k\right) \leq 0, \quad \forall k \in [K]
\end{align}
Here $\Pi \subseteq \mathbb{R}^{\Bar{n}}$ is a closed and convex set, $\pi \in \Pi$ is a vector of decision parameters, and $c$ is a constant vector. Further, $f\left(\pi, \delta\right)$ is a continuous and convex function in $\pi$, while $\delta_i$ for $k \in [K]$ is an independent and identically distributed (i.i.d.) random sample drawn from the uncertainty set $\Delta$. An example of set $\Delta$ and random vector $\delta$ in the context of the \acrshort{otr} problem are given in Section~\ref{uncertainty-set}.

In recent years, research on scenario optimization has focused on generalization guarantees for Problem~\eqref{eq:classic-scenario} to extend its solution properties to unseen samples in the probability space. To elucidate the generalization guarantee we leverage in this work, we first introduce some key definitions. First, we define the risk of a vector of decision variables $\pi$ as the probability that constraints of the form $f(\pi, \delta) \leq 0$ are violated.
\begin{definition} (Risk) \label{def:prob-violation-classical}
The \textit{risk} of a given decision parameter $\pi \in \Pi$, $V(\pi)$ is the probability with which constraints $f(\pi, \delta) \leq 0$ are violated, i.e., $V(\pi) = \mathbb{P}\{\delta \in \Delta: f(\pi, \delta) > 0 \}$.
\end{definition}

In addition to the objective $c^T \pi^*$ for the optimal decision variables $\pi^*$ corresponding to the solution of Problem~\eqref{eq:classic-scenario}, the risk $V(\pi^*)$ is a fundamental quantity to evaluate the quality of a vector of decision parameters $\pi^*$. While the problem of providing bounds on the number of samples $K$ that are necessary to provide a solution with a small risk have been studied extensively~\citep{scenario-main,campi2008exact}, these bounds are often not tight for problems that are not fully supported, i.e., problems where the removal of some constraints will not change the solution $\pi^*$ to Problem~\eqref{eq:classic-scenario}. As a result, to obtain bounds on the risk of a vector of optimal decision parameters, we adopt the approach used in~\cite{risk-complexity-scenario}, which relates the level of risk of the optimal solution $\pi^*$ of Problem~\eqref{eq:classic-scenario} to its number of support constraints, formally defined as follows.


\begin{definition} (Support Constraint) \label{def:support-constraint}
A constraint $f(\pi, \delta) \leq 0$ of Problem~\eqref{eq:classic-scenario} is a support constraint if its removal changes the optimal solution $\pi^*$ of Problem~\eqref{eq:classic-scenario}.
\end{definition}




Using the above notion of a support constraint,~\cite{risk-complexity-scenario} establish the following key result.

\begin{proposition} (Relation between Number of Support Constraints and Risk, see~\cite{risk-complexity-scenario}) \label{prop:scenario-classic-theorem}
Fix a parameter $\beta \in (0, 1)$ and consider the following polynomial equation in $t$ for any $k \in \{0, 1, ..., \Bar{n} \}$

\begin{align} \label{eq:poly-eq}
     {K \choose k} t^{N-k}-\frac{\beta}{2 K} \sum_{i=k}^{K-1} {i \choose k} t^{i-k}-\frac{\beta}{6 K} \sum_{i=N+1}^{4 K} {i \choose k} t^{i-k}=0
\end{align}
This equation has exactly two solutions $\underline{t}(k), \Bar{t}(k) \in [0, \infty)$ where $\underline{t}(k) \leq \Bar{t}(k)$. Let $\pi^*$ be the unique solution to Problem~\eqref{eq:classic-scenario}. Then, when the number of samples is greater than the number of decision variables, i.e., $K>\Bar{n}$, 
\begin{align*}
    \mathbb{P}\left\{\underline{\epsilon}\left(s^{*}\right) \leq V\left(\pi^{*}\right) \leq \bar{\epsilon}\left(s^{*}\right)\right\} \geq 1-\beta,
\end{align*}
where $s^*$ are the number of support constraints, $\bar{\epsilon}(s^*) = 1-\underline{t}(s^*)$ and $\underline{\epsilon}\left(s^{*}\right) = \max \{ 0, 1-\Bar{t}(s^*)\}$.
\end{proposition}


The guarantee provides upper and lower bounds on the level of risk of the optimal solution $\pi^*$ to Problem~\eqref{eq:classic-scenario} with a level of confidence $1-\beta$. These bounds can be computed by assessing the number of support constraints $s^*$ after computing the optimal solution to Problem~\eqref{eq:classic-scenario}. In general, we can take the parameter $\beta$ to be arbitrarily close to zero, since the solutions of Equation~\eqref{eq:poly-eq} depend logarithmically on $\beta$. We further note that the upper bound on the level of risk typically increases with the number of support constraints. Thus, to provide strong generalization guarantees, we seek decision variables of Problem~\eqref{eq:classic-scenario} that simultaneously lead to a small objective $c^T \pi^*$ and a relatively small number of support constraints.

\subsection{Data-Driven Online Traffic Routing} \label{data-driven-oad}

In this and the following sections, we leverage the scenario optimization paradigm to devise a time-independent data-driven algorithm for the \acrshort{otr} problem whose performance can be formally generalized to unseen instances through Theorem~\ref{prop:scenario-classic-theorem}. The data-driven algorithm has two phases: (i) an offline parameter learning phase and (ii) an online decision-making phase. In this section, we present the general structure of the offline problem that we solve to learn the data-driven algorithm's parameters. Then, we develop a specific parametrization of the data-driven algorithm in Section~\ref{offline-parameter-learning} and show that it is a special case of the scenario optimization Problem~\eqref{eq:classic-scenario} in Section~\ref{data-driven-online-alg}.

We aim to devise an algorithm $\mathbf{A}^*$ for \acrshort{otr} that achieves the minimum possible competitive ratio for all input sequences $\mathcal{I} \in \Delta$. In other words, the objective is to find $\mathbf{A}^*$ that satisfies:
\begin{argmini!}|s|[2]                   
    {\mathbf{A}}                               
    {\alpha \label{eq:OPT-Obj-Full-Data}}   
    {\label{eq:Eg001-Full-Data}}             
    {\mathbf{A}^* =}                                
    \addConstraint{ALG(\mathcal{I}, \mathbf{A})}{\leq \alpha OPT(\mathcal{I}), \quad \forall \mathcal{I} \in \Delta \label{eq:OPTcon1-Full-Data}}    
    \addConstraint{\mathbf{A} \text{ satisfies~\eqref{eq:prob-constraint}-\eqref{eq:capacity-constraints} }, \quad \forall \mathcal{I} \in \Delta}{ \label{eq:OPTcon2-Full-Data}}    
\end{argmini!}
where $\alpha$ is the competitive ratio. Note that to appropriately define the set $\Delta$, we fix the number of arriving users $n$ in the network. Problem~\eqref{eq:OPT-Obj-Full-Data}-\eqref{eq:OPTcon2-Full-Data} is equivalent to the robust optimization problem: 
\begin{align} \label{eq:min-max}
    \mathbf{A}^* = \argmin_{\mathbf{A}: \mathbf{A} \text{ satisfies~\eqref{eq:prob-constraint}-\eqref{eq:capacity-constraints} }, \forall \mathcal{I} \in \Delta} \max_{\mathcal{I} \in \Delta} \frac{ALG(\mathcal{I}, \mathbf{A})}{OPT(\mathcal{I})}
\end{align}
Solving Problem~\eqref{eq:OPT-Obj-Full-Data}-\eqref{eq:OPTcon2-Full-Data} and~\eqref{eq:min-max} is typically intractable, since the problem is discrete, due to the binary constraints, and we are optimizing over the universe of all possible algorithms. To obtain a tractable surrogate problem we consider two simplifications. First, we restrict the algorithm $\mathbf{A}^*$ to lie in a parametric class, which we denote as $\mathcal{A}$. We elaborate on our choice of $\mathcal{A}$ in Section~\ref{offline-parameter-learning}. Second, we allow fractional allocations by relaxing the binary constraints with:
\begin{align} \label{eq:lp-relaxed-constraint}
    x_{ia} \geq 0, \quad \forall i\in [n], \forall a \in [M]
\end{align}
We thus compute $OPT(\mathcal{I})$ for each input sequence $\mathcal{I}$ as the optimal objective of the problem of minimizing~\eqref{eq:OPT-integer-Obj} subject to the constraints~\eqref{eq:prob-constraint},~\eqref{eq:capacity-constraints} and~\eqref{eq:lp-relaxed-constraint}. Note that the optimal objective value of this fractional problem serves as a lower bound on the objective of the integral Problem~\eqref{eq:OPT-integer-Obj}-\eqref{eq:capacity-constraints}. Further, note that the fractional problem is a linear program with $n + 2n M$ constraints and $n M$ decision variables. Since the size of this linear program is polynomial in the inputs defined by $\mathcal{S}$ and $\mathcal{I}$, it follows that $OPT(\mathcal{I})$ can be computed in polynomial time for any input sequence $\mathcal{I}$.

Given the above relaxations, the objective of learning algorithm $\mathbf{A}^*$ can be summarized as follows:
\begin{argmini!}|s|[2]                   
    {\mathbf{A}}                               
    {\alpha \label{eq:OPT-Obj-large-Data}}   
    {\label{eq:Eg001-large-Data}}             
    {\mathbf{A}^* = }                                
    \addConstraint{ALG(\mathcal{I}, \mathbf{A})}{\leq \alpha OPT(\mathcal{I}), \quad \forall \mathcal{I} \in \Delta \label{eq:OPTcon1-large-Data}}    
    \addConstraint{\mathbf{A} \text{ satisfies~\eqref{eq:prob-constraint}, \eqref{eq:capacity-constraints}, \eqref{eq:lp-relaxed-constraint} }, \quad \forall \mathcal{I} \in \Delta}{ \label{eq:OPTcon2-large-Data}}    
    \addConstraint{\mathbf{A}}{\in \mathcal{A} \label{eq:OPTcon3-large-Data}}
\end{argmini!}
We note that Problem~\eqref{eq:OPT-Obj-large-Data}-\eqref{eq:OPTcon3-large-Data} is not equivalent to Problem~\eqref{eq:OPT-Obj-Full-Data}-\eqref{eq:OPTcon2-Full-Data}, and so we introduce a rounding procedure to ensure feasibility for Problem~\eqref{eq:OPT-Obj-Full-Data}-\eqref{eq:OPTcon2-Full-Data}. Doing so enables us to reason about the properties of the original Problem~\eqref{eq:OPT-Obj-Full-Data}-\eqref{eq:OPTcon2-Full-Data}. We further note that even the surrogate problem may be intractable since the probability space could have infinitely many samples resulting in infinite constraints. To efficiently solve the surrogate problem, we use a constraint sampling procedure wherein we learn the optimal parametrization through $K$ i.i.d. input sequences $\mathcal{I}_1, ..., \mathcal{I}_K$ of the \acrshort{otr} problem. Specifically, we learn algorithm $\mathbf{A}^*$ through the solution of the following scenario-based problem:

\begin{argmini!}|s|[2]                   
    {\mathbf{A} \in \mathcal{A}}                               
    {\alpha \label{eq:OPT-Obj-K-Data}}   
    {\label{eq:Eg001-K-Data}}             
    {\mathbf{A}^* = }                                
    \addConstraint{ALG(\mathcal{I}_k, \mathbf{A})}{\leq \alpha OPT(\mathcal{I}_k), \quad \mathcal{I}_k \in \Delta, \forall k \in [K] \label{eq:OPTcon1-K-Data}}    
    \addConstraint{\mathbf{A} \text{ satisfies~\eqref{eq:prob-constraint}, \eqref{eq:capacity-constraints}, \eqref{eq:lp-relaxed-constraint} }, \quad \mathcal{I}_k \in \Delta, \forall k \in [K]}{ \label{eq:OPTcon2-K-Data}}    
\end{argmini!}

The optimal objective $\alpha^*$ of the above problem is the maximum ratio between the total travel cost of the online solution of $\mathbf{A}^*$ and that of the optimal offline solution for $K$ training input sequences. 


\subsection{Offline Parameter Learning for Time-independent Algorithm} \label{offline-parameter-learning}
In this section, we propose a parametric class $\mathcal{A}$ over which we optimize to learn the routing algorithm $\mathbf{A}^*$. To develop strong generalization guarantees to unseen sequences using the scenario approach (Theorem~\ref{prop:scenario-classic-theorem}) we must choose a parametric class $\mathcal{A}$ that reduces the offline learning Problem~\eqref{eq:OPT-Obj-K-Data}-\eqref{eq:OPTcon2-K-Data} to a special case of the scenario optimization Problem~\eqref{eq:classic-scenario}. Furthermore, we seek to develop a parametric class $\mathcal{A}$ that is sufficiently simple for optimization, i.e., Problem~\eqref{eq:OPT-Obj-K-Data}-\eqref{eq:OPTcon2-K-Data} can be computed quickly, by finding a low-dimensional parametrization while also being rich enough to provide good performance on unseen sequences.


The class of algorithms $\mathcal{A}$ we consider can be parametrized by a vector $\mathbf{p}_{\theta} \in \mathbb{R}^M$ for each value-of-time $\theta$. Here $p_{\theta a}$ for each $a \in [M]$ represents the probability that a user with value-of-time $\theta$ is allocated to arc $a$. We can learn the optimal $\mathbf{p}_{\theta}$ through the solution of the following problem.
\begin{argmini!}|s|[2]                   
    {\substack{p_{\theta a}, \\ \forall \theta \in \Theta, a \in [M]}}                               
    {\alpha \label{eq:OPT-Obj-offline}}   
    {\label{eq:Eg001-offline}}             
    {\mathbf{A}^* = }                                
    \addConstraint{\sum_{i=1}^{n} \sum_{a = 1}^M p_{\theta_i a} \theta_i t_j}{\leq \alpha OPT(\mathcal{I}_{\Bar{k}}), \quad \forall \Bar{k} \in [K] \label{eq:OPTcon1-offline}}    
    \addConstraint{p_{\theta a} }{\geq 0, \quad \forall \theta \in \Theta, \forall a \in [M] \label{eq:OPTcon2-offline}}    
    \addConstraint{\sum_{a = 1}^{M} p_{\theta a} }{= 1, \quad \forall \theta \in \Theta \label{eq:OPTcon3-offline}}    
    \addConstraint{\sum_{i = 1}^{n} \mathbbm{1}_{\tau_{l}^{(\Bar{k})} \in [\tau_{i}^{(\Bar{k})}, \tau_{i}^{(\Bar{k})} + t_a]} p_{\theta_i a} }{\leq c_a, \forall a \in [M], \forall \tau_{l}^{(\Bar{k})} \in \{\tau_{1}^{(\Bar{k})}, ..., \tau_{n}^{(\Bar{k})} \}, \forall \Bar{k} \in [K] \label{eq:OPTcon4-offline}}    
\end{argmini!}
Here,~\eqref{eq:OPTcon1-offline} are competitive ratio constraints ensuring that the total cost of the optimal parametrization is no more than $\alpha$ times the cost of the optimal offline solution,~\eqref{eq:OPTcon2-offline} are non-negativity constraints,~\eqref{eq:OPTcon3-offline} are assignment constraints, and~\eqref{eq:OPTcon4-offline} are capacity constraints. Further, $\tau_{i}^{(\Bar{k})}$ is the arrival time of agent $i$ in input sequence $\Bar{k} \in [K]$. Note that this parametrization only depends on users' value-of-time and is independent of the traffic states or the time at which a given user arrives. The dependency on the value-of-time of the optimal parametrization ensures that the allocation decision is consistent across users belonging to a given value-of-time. The above parametrization can be naturally extended to a time-dependent setting by learning allocation probabilities $\mathbf{p}_{\theta}^{j}$ for different time intervals $j$ at which users arrive. We explore such a parametrization in more detail in Section~\ref{data-driven-online-alg}. We note that one could also choose parametrizations that depend on the traffic state, which would give the parametrization more power in catering specific allocation decisions based on the set of possible traffic states. However, since the number of traffic states can typically be large and is a function of user arrival rates, accounting for traffic states would result in a high-dimensional parametrization without giving much more decision making power than the time-dependent algorithm (see Section~\ref{data-driven-online-alg}) that adjusts allocations based on changes in user arrival rates. 
Our low-dimensional parametrization ensures that Problem~\eqref{eq:OPT-Obj-offline}-\eqref{eq:OPTcon4-offline} is tractable, which may not be possible with more complex parametrizations. Problem~\eqref{eq:OPT-Obj-offline}-\eqref{eq:OPTcon4-offline} is a linear program with $|\Theta| M$ decision variables and $|\Theta| M + |\Theta| + K + n M K$ constraints.

Algorithm~\ref{alg:Data-Driven} sketches our data-driven online algorithm that proceeds in two phases. It first learns the optimal parametrization $\mathbf{p}_{\theta}$ for each $\theta \in \Theta$ through the solution of the offline Problem~\eqref{eq:OPT-Obj-offline}-\eqref{eq:OPTcon4-offline}. These learned parameters are then used to make online decisions for each test sequence as follows: when a user with a value-of-time $\theta$ arrives, the user is allocated to arc $a$ with probability $p_{\theta a}$.
\begin{algorithm} 
\SetAlgoLined
\SetKwInOut{Input}{Input}\SetKwInOut{Output}{Output}
 \textbf{PHASE I: Parameter Learning} \\
 \For{$k = 1, ..., K$}{
    $OPT_k \leftarrow $ Solution of fractional \acrshort{otr} for input sequence $\mathcal{I}_k$ \;
 }
 $\mathbf{p}_{\theta}^* \leftarrow \argmin $~\eqref{eq:OPT-Obj-offline} s.t. \eqref{eq:OPTcon1-offline}-\eqref{eq:OPTcon4-offline} for each type $\theta$  \;
 \textbf{PHASE II: Online Decision Making} \\
 \textbf{When User $i \in [n]$ in Test Sequence arrives:} \\
 \quad \quad Allocate $i$ to arc $a$ with probability $p_{\theta_i a}^*$
\caption{Data-driven Algorithm}
\label{alg:Data-Driven}
\end{algorithm}

Note that the allocation probabilities we learn in Algorithm~\ref{alg:Data-Driven} serve as a fractional heuristic for the corresponding discrete allocations, which must be made in the context of the \acrshort{otr} problem. That said, we expect that as the number of users increases, the online solution corresponding to the discrete allocations will converge to that of the fractional allocations by the law of large numbers.


\subsection{Generalization Guarantee for Data-Driven Algorithm} \label{data-driven-online-alg}
We now present a probabilistic generalization guarantee for Algorithm~\ref{alg:Data-Driven} that learns the optimal parametrization of Problem~\eqref{eq:OPT-Obj-offline}-\eqref{eq:OPTcon4-offline}. In particular, Corollary~\ref{cor:scenario-cor} establishes an explicit bound on the level of risk based on the number of support constraints of the offline learning Problem~\eqref{eq:OPT-Obj-offline}-\eqref{eq:OPTcon4-offline}.


\begin{corollary} (Theoretical Bounds on Risk) \label{cor:scenario-cor}
Fix a parameter $\beta \in (0, 1)$ and suppose that the solution to the Problem~\eqref{eq:OPT-Obj-offline}-\eqref{eq:OPTcon4-offline} is unique for instances $\mathcal{I}_k$ for $k \in [K]$. Let $s^*$ be the number of support constraints of Problem~\eqref{eq:OPT-Obj-offline}-\eqref{eq:OPTcon4-offline} and let $\underline{\epsilon}\left(s^{*}\right), \bar{\epsilon}(s^*)$ be computed through the solutions $\underline{t}(s^*), \Bar{t}(s^*)$ of
\begin{align} \label{eq:poly-scenario-otr}
     {K \choose s^*} t^{N-s^*}-\frac{\beta}{2 K} \sum_{i=s^*}^{K-1} {i \choose s^*} t^{i-s^*}-\frac{\beta}{6 K} \sum_{i=N+1}^{4 K} {i \choose s^*} t^{i-s^*}=0
\end{align}
as in Theorem~\ref{prop:scenario-classic-theorem}. Then, it follows that
\begin{align*}
    \mathbb{P}\left\{\underline{\epsilon}\left(s^{*}\right) \leq V\left(\pi^{*}\right) \leq \bar{\epsilon}\left(s^{*}\right)\right\} \geq 1-\beta.
\end{align*}
\end{corollary}
\begin{proof}{Proof}
To prove this claim, we aim to establish that Problem~\eqref{eq:OPT-Obj-offline}-\eqref{eq:OPTcon4-offline} is a special case of Problem~\ref{eq:classic-scenario}. We first note that Constraints~\eqref{eq:OPTcon2-offline} and~\eqref{eq:OPTcon3-offline} are independent of the problem instance and form a closed and convex set, analogous to the closed and convex set~$\Pi$ in the scenario optimization Problem~\eqref{eq:classic-scenario}. For this problem we have that $\Pi \subseteq \mathbb{R}^{|\Theta| M}$. However, due to the constraint $\sum_a p_{\theta a} = 1$, the optimal parametrization lies on a hyperplane of dimension $|\Theta| (M-1)$. Next, there is one competitive ratio constraint~\eqref{eq:OPTcon1-offline} and $n$ capacity constraints~\eqref{eq:OPTcon4-offline} for each instance, resulting in a total of $n+1$ constraints of the form: $f_{1}(\textbf{p}, \mathcal{I}) \leq 0, \ldots, f_{n+1}(\textbf{p}, \mathcal{I}) \leq 0$. Note that these constraints reduce to a single scalar valued constraint function in the scenario optimization Problem~\eqref{eq:classic-scenario} since $f(\textbf{p}, \mathcal{I}) \doteq \max _{l=1, \ldots, n + 1} f_{l}(\textbf{p}, \mathcal{I})$. Since the constraints are convex and the objective function is linear, we have that Problem~\eqref{eq:OPT-Obj-offline}-\eqref{eq:OPTcon4-offline} is a special case of the scenario optimization Problem~\ref{eq:classic-scenario}. Thus, the result follows through a direct application of Theorem~\ref{prop:scenario-classic-theorem}.
\end{proof}

This result provides a probabilistic generalization guarantee for the online algorithm that utilizes the solution of the offline learning Problem~\eqref{eq:OPT-Obj-offline}-\eqref{eq:OPTcon4-offline} on unseen test instances. The key to establishing the above result lies in the choice of the parametric class $\mathcal{A}$ that yields a problem that is conducive to the application of the scenario approach generalization guarantee (Theorem~\ref{prop:scenario-classic-theorem}). To ensure that the learning problem is sufficiently simple for optimization, we developed a parametric class with a small number of decision variables. 

Finally, we note that the competitive ratio obtained through the solution of Problem~\eqref{eq:OPT-Obj-offline}-\eqref{eq:OPTcon4-offline} provides an upper bound on the best achievable competitive ratio since we have chosen a specific algorithm parametrization in Problem~\eqref{eq:OPT-Obj-offline}-\eqref{eq:OPTcon4-offline}. 

\paragraph{Extension to Time-Dependent Algorithm.} 


In the following, we elucidate an extension of Algorithm~\ref{alg:Data-Driven} to the time-dependent setting where the probability of allocating users varies with the change in the user arrival rates. The motivation for considering the time-dependent setting stems from the time dependency of the arrival process of users in typical traffic routing applications. Thus, we consider the class of algorithms $\mathcal{A}'$ parametrized by vectors $\mathbf{p}_{\theta}^j \in \mathbb{R}^{M}$ for each value-of-time $\theta \in \Theta$ and time interval $j \in [q]$ for some $q \in \mathbb{N}$. Here $p_{\theta a}^j$ for each arc $a \in [M]$ and time interval $j \in [q]$ represents the probability that a user with a value-of-time of $\theta$ is allocated to arc $a$ during time interval $j$. Note that we can write a linear program similar to Problem~\eqref{eq:OPT-Obj-offline}-\eqref{eq:OPTcon4-offline} to compute the optimal allocation probabilities $\mathbf{p}_{\theta}^j$ offline. The learned parameters can then be used to make online allocations in an analogous manner to that of the time-independent data-driven algorithm presented in Algorithm~\ref{alg:Data-Driven}. In particular, when a user $i$ with a value-of-time $\theta_i$ arrives during time interval $j$, i.e., $\tau_i \in [\mu_j, \mu_{j+1})$ for $0 \leq \mu_j < \mu_{j+1}$, then the user is allocated to arc $a$ with probability $p_{\theta_i a}^j$. Finally, we note the following key observations regarding the time-dependent parametrization that highlight the risk and offline objective value properties in this setting.

\begin{remark}
Following a similar reasoning to that used in the proof of Corollary~\ref{cor:scenario-cor}, the linear program used to compute the allocation probabilities $\mathbf{p}_{\theta}^j$ can be shown to be a special case of the scenario optimization Problem~\eqref{eq:classic-scenario}. As a result, the theoretical bounds on the risk obtained in Corollary~\ref{cor:scenario-cor} extend to this setting as well.
\end{remark}

\begin{remark} \label{rmk:opt-obj-comparison}
The optimal objective value $\alpha^*_\mathcal{T}$ of the linear program to compute the time-dependent allocation probabilities is at most the optimal objective value $\alpha^*$ of the linear Program~\eqref{eq:OPT-Obj-offline}-\eqref{eq:OPTcon4-offline} to compute time-independent allocation probabilities. This follows since a special case of the time-dependent allocation probabilities corresponds to the time-independent setting when the allocation probabilities for the different time intervals are exactly the same, i.e., $\mathbf{p}_{\theta}^j = \mathbf{p}_{\theta}^{j'}$ for all time intervals $j, j' \in [q]$. As a result, the time-dependent allocation achieves superior performance as compared to the time-independent allocation when comparing the objective value of their corresponding offline linear programs. That is, the empirical competitive ratio on the training instances for the time-dependent allocation probabilities can never be more than that corresponding to the time-independent allocation probabilities, i.e., $\alpha^*_{\mathcal{T}} \leq \alpha^*$.
\end{remark} 

Remark~\ref{rmk:opt-obj-comparison} highlights that the time-dependent algorithm is guaranteed to have a superior performance as compared to the time-independent algorithm on the training set. However, the generalization of the in-sample performance guarantees of both algorithms to unseen sequences depends on the number of support constraints, which we determine through numerical experiments in the next section.



\section{Numerical Comparison of Data Driven and Greedy Algorithms} \label{main-numerical}
We now evaluate the performance of the data-driven approach for both the time-independent (Algorithm~\ref{alg:Data-Driven}) and time-dependent (see Section~\ref{data-driven-online-alg}) parameter learning settings through numerical experiments. To this end, we first numerically validate the theoretical bounds on the level of risk obtained in Corollary~\ref{cor:scenario-cor} through an analysis of the number of support constraints corresponding to the offline parameter learning problem. Furthermore, to gauge the efficacy of the data-driven algorithms, we compare their performance to that of the greedy algorithm by analyzing the distribution of the ratio between the algorithm cost to that of the optimum offline solution on unseen input sequences. The interest for comparing the performance of these algorithms to the greedy algorithm is both theoretical and practical. On the theoretical front, the greedy algorithm serves as a benchmark since it has the lowest ratio between the online and offline solutions amongst all deterministic algorithms for a two-arc parallel network. On the applied front, state-of-the-art traffic navigational applications often route users using a greedy algorithm \citep{Cabannes2017TheIO}.


In the following, we first define the uncertainty set on which we built our input sequences $\mathcal{I}$ (Section~\ref{uncertainty-set}). Then, we introduce our experimental design, based on an application case of the San Francisco Bay Area, CA (Section~\ref{experiment-design}). Finally, we present numerical results (Section~\ref{results-numerical}) for both this application setting and a range of synthetic scenarios to investigate the performance of the data-driven and greedy algorithms for different user arrival profiles that may be encountered in practice.

\subsection{Definition of the Uncertainty Set and Time-Dependent Data-Driven Algorithm} \label{uncertainty-set}
Given a fixed parallel transportation graph with specified arc capacities and travel times, the sources of uncertainty that completely define an input sequence are the users' arrival times and their values of time. Thus, we represent an input sequence $\mathcal{I}$ as an $n \times 2$ dimensional matrix with an arrival time and a value-of-time for each of the $n$ users. To model typical traffic routing applications wherein the user arrival rate varies at different times of the day we create input sequences as follows.
\begin{enumerate}
    \item When a user arrives into the system, she has a value-of-time $\theta \in \Theta$ with a known probability $\Tilde{p}_{\theta}$. We further have that $\sum_{\theta \in \Theta}\Tilde {p}_{\theta} = 1$.
    \item In each time interval $j$ lying in the set $\mathcal{T} = \{[\mu_1, \mu_2), [\mu_2, \mu_3), \ldots, [\mu_q, \mu_{q+1})\}$, where $\mu_1 = 0$ and $\mu_{q+1} = \infty$, users arrive according to a Poisson process with a rate $\lambda_{j}$, i.e., the inter-arrival times between subsequent user arrivals in the time period $[\mu_{j}, \mu_{j+1})$ is exponentially distributed with a rate $\lambda_{j}$.
\end{enumerate}

\subsection{Experimental Design} \label{experiment-design}
The performance of both the data-driven algorithms as well as the greedy algorithm depends on the specific parameters of the underlying problem instance. To design a numerical experiment that reflects a realistic application, we proceed as follows: first, we derive a representative instance based on real-world data. Then, to solve the offline linear programs corresponding to both data-driven algorithms we transform these instances to create a scaled representation of the real-world scenarios. Note that scaling travel times and users' values of time is without loss of generality because the ratio $\frac{ALG(\mathcal{I}, \mathbf{A})}{OPT(\mathcal{I})}$ will remain unchanged as the objective function of Problem~\eqref{eq:OPT-integer-Obj}-\eqref{eq:capacity-constraints} is linear. Next, we reduce the size of the instances to 120 users only for the purposes of the case study so that we can quickly generate solutions for many scenarios with low computational effort. 
We note that we can also train the algorithm on larger sized instances, e.g., 100 instances with 500 users, within a time span of a few minutes. Furthermore, we note that due to the nature of the algorithm parametrization, the algorithm can readily be applied online to larger test instances without losing its computational efficiency. We further expect by the law of large numbers that our results in the scaled setting will provide good insights for larger problem instances, since user arrivals are stochastic. Finally, we create additional instances by varying the arrival profiles of users to create other meaningful scenarios.



\subsubsection{Representative scenarios.} We derive representative instances based on data from the the San Francisco Bay Area for travel on major highways (Figure~\ref{highway-map}). In particular, we consider three alternative routes, which we derive based on traffic data from TomTom \citep{tom-tom-data}, and select the routes with the minimum (green route), median (blue route), and maximum (red route) travel times between the respective origin-destination pair. Table~\ref{tab: data-scenario-1} presents the travel times of these routes. We note here that the parallel network abstraction is generally a good first-order approximation for highway commuting~\citep{caltrans}, which is considered in the traffic routing application depicted in Figure~\ref{highway-map}.

To derive user arrival rates, we used data from the Caltrans PeMS database\footnote{\url{https://pems.dot.ca.gov/}} to obtain representative user flow rates at different points along the sections of the routes corresponding to the three highways depicted in Figure~\ref{highway-map}. The user flow rates presented in Figure~\ref{city-center-map} are derived for the morning rush on Tuesday, July 6, 2021 for the following sections of the three freeways - (i) South Airport Blvd. for the US-101N freeway (Route 1), (ii) Westborough Blvd. for the I-280N freeway (Route 2), and (iii) Tennyson Rd. for the I-880 freeway (Route 3). We note that the user flow rate data represents a typical pattern of user flows during the morning rush for each of the three freeways.


\begin{figure}[tbh!]
  \centering
  \begin{minipage}[b]{0.48\textwidth}
    \includegraphics[width=\textwidth, height=6cm]{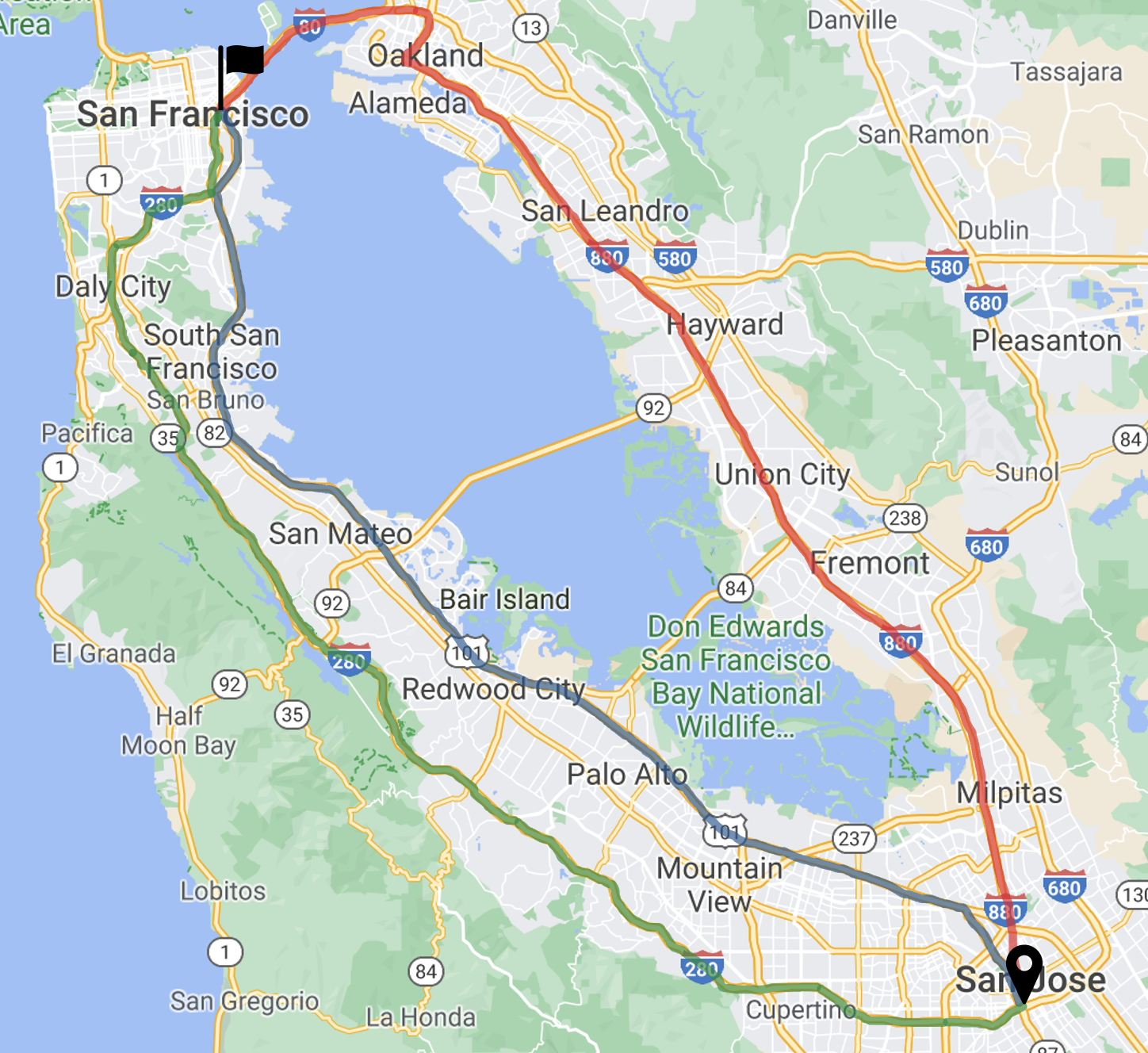}
    \caption{Three selected routes from San Jose to San Francisco for the highway application. The destination is marked with a flag and the origin is marked with a map pin.}
    \label{highway-map}
  \end{minipage}
  \hfill
  \begin{minipage}[b]{0.48\textwidth}
    \definecolor{mycolor1}{rgb}{0.00000,0.44700,0.74100}%
\definecolor{mycolor2}{rgb}{0.85000,0.32500,0.09800}%
\definecolor{mycolor3}{rgb}{0.1,0.7,0.2}%
\begin{tikzpicture}

\begin{axis}[%
width=1.8in,
height=1.6in,
at={(0in,0in)},
scale only axis,
xmin=0,
xmax=10,
xtick={0, 2, 4, 6, 8, 10},
xticklabels={{05:00},{06:00},{07:00},{08:00},{09:00},{10:00}},
xticklabel style={rotate=30},
xlabel style={font=\color{white!15!black}},
xlabel={Time of Day},
ymin=0,
ymax=7000,
ylabel style={font=\color{white!15!black}},
ylabel={Flow (veh/h)},
axis background/.style={fill=white},
legend style={legend cell align=left, align=left, draw=white!15!black, at={(1,0.4)}}
]
\addplot[const plot, color=mycolor1, line width=2pt] table[row sep=crcr] {%
0	3356\\
2	5160\\
4	6011\\
6	6464\\
8	5950\\
10	5950\\
};
\addlegendentry{Route 1}

\addplot[const plot, color=mycolor3, line width=2pt] table[row sep=crcr] {%
0	1101\\
2	2334\\
4	3593\\
6	3832\\
8	3511\\
10	3511\\
};
\addlegendentry{Route 2}

\addplot[const plot, color=mycolor2, line width=2pt] table[row sep=crcr] {%
0	2882\\
2	4257\\
4	4596\\
6	4991\\
8	4487\\
10	4487\\
};
\addlegendentry{Route 3}

\end{axis}

\begin{axis}[%
width=\textwidth,
height=1.2in,
at={(0in,0in)},
scale only axis,
xmin=0,
xmax=1,
ymin=0,
ymax=1,
axis line style={draw=none},
ticks=none,
axis x line*=bottom,
axis y line*=left,
legend style={legend cell align=left, align=left, draw=white!15!black}
]
\end{axis}
\end{tikzpicture}%
 \vspace{-20pt}
    \caption{Flow of users during the weekday morning rush for the three highways corresponding to the routes depicted in Figure~\ref{highway-map}. Data is obtained for Tuesday, July 6, 2021.}
    \label{city-center-map}
  \end{minipage}
\end{figure}

\subsubsection{Parameter transformation.} We now transform both scenarios into a scaled representation by (i) scaling users' values of time and the arc travel times and (ii) adjusting the arc capacity and user arrival rates based on the scaled travel times. Our transformation is as follows.

To derive representative values of time for different users, we use the minimum, median, and maximum income levels in the San Francisco Bay Area \cite{SFUrgency}. In particular, we consider three user types and normalize each user type's value-of-time as shown in Table~\ref{tab: urgency-vals}, such that 
$$\theta_1 = 1,\quad \theta_2 = \frac{\text{Median Income}}{\text{Minimum Income}},\quad \theta_3 = \frac{\text{Maximum Income}}{\text{Minimum Income}}.$$
Note that we take the income of users as a surrogate representation of user's values of time. While a range of other value-of-time representations could have been chosen, we assume a proportionality between the value-of-time and income level of users for simplicity. Doing so enables us to capture a range of user values of time that are likely to be observed in practice since values of time tend to increase with the income level of users~\citep{thomas1970value}. 
We further assume that the share of the population belonging to each value-of-time is given by specified quantile cutoffs. We set these cutoffs based on the average of the minimum and median incomes and the average of the median and maximum incomes. We round the data in Table~\ref{tab: urgency-vals} to the nearest \$10000 and calculate the fraction of users in each quantile range using Figure 7 of \cite{SFUrgency}.
\begin{table}[tbh!] 
\centering
\caption{Values of Time of Users in San Francisco Bay Area, CA }
\footnotesize
\begin{tabular}{lccc}
\toprule
  & Minimum & Median & Maximum \\
\midrule
Household Income SF [\$] & 10000 & 90000 & 200000   \\
Value-of-Time [\$/Time Unit] & 1 & 9  & 20   \\
Quantile Range [\$] & [10000, 50000] & [50000, 150000] & [150000, 200000]\\
Fraction of Users & 0.32 & 0.39 & 0.29 \\
\bottomrule
\end{tabular} \label{tab: urgency-vals}
\end{table}

For the numerical experiments, we consider a population of 120 users. Accordingly, we set the capacity of the slowest route to 120 (see Table~\ref{tab: data-scenario-1}) to ensure the system's feasibility and scale the real-world travel times such that the ratio of the costs $t_j$ closely matches the ratios of the real-world travel times for each scenario. Finally, we set the user arrival rate and throughput for the first two routes. To do so, we first note that the first two routes are the predominantly used ones, as their travel times are much lower compared to the third route. Since we have two degrees of freedom, arrival rate and vehicle throughput, we fix the vehicle throughput on the first two routes to one vehicle per time unit. Note that setting the same throughput value for the first two routes is reasonable since the highways connecting San Jose and San Francisco have similar capacities~\citep{pems-database}. To achieve this throughput, we set capacities as given in Table~\ref{tab: data-scenario-1}.

\begin{table}[tbh!]
\centering
\caption{Cost and Capacity Parameters of Input instance for the highway scenario.}
\footnotesize
\begin{tabular}{lccc}
\toprule
   Route     & Green & Blue & Red \\ 
\midrule
Travel Time [s] & 2660 & 3200 & 17560 \\ 
Cost ($t_j$) in Model [Time Units] & 20                    & 24                     & 130               \\ 
Capacity ($c_j$) in Model [Number of Users] & 20                    & 24                    & 100         \\ 
\bottomrule
\end{tabular} \label{tab: data-scenario-1}
\end{table}

Next, to set the arrival rates, we consider the period of the morning rush on weekdays wherein the throughput pattern can be modelled as step function with constant user arrival rates during each time period, cf. Figure~\ref{city-center-map}. In particular, we average the vehicle flows depicted in Figure~\ref{city-center-map} across the three routes for each hour interval and set the flow rate of vehicles between the time interval $5-6$a.m. as $1.2$ vehicles per time unit and scale the width of each time interval to represent fourteen time units, i.e., $\mu_{j+1} - \mu_j = 14$ for all $j \in \{1, 2, 3, 4 \}$. Note that this represents an under-congested condition since the arrival rate of users in this time period is lower than the combined throughput of the first two routes. We then scale the remaining flow rates for the other time intervals as:
\begin{align*}
    \lambda_j = \frac{\text{Average Flow rate in time interval j}}{\text{Average Flow rate between } 5-6\text{a.m.}} \times 1.2
\end{align*}


\subsubsection{Additional Scenarios.} Since the above scenario only captures a limited set of parameters, we additionally explore a broader range of arrival rates that capture traffic patterns beyond the one considered above. In particular, we model the flow of traffic through the course of an entire day by modelling both the morning and afternoon peaks and consider the following arrival rate scenarios as depicted in Table~\ref{tab:synthetic-scenarios}. In particular, the five scenarios represent (i) a constant arrival rate, (ii) a high rush hour arrival rate, (iii) a low rush hour arrival rate, (iv) a higher afternoon peak arrival rate, and (v) a higher morning peak arrival rate. For completeness, we also summarize the arrival rate profile of the above described highway scenario.


\begin{table}[tbh!] 
\centering
\caption{Arrival Rate profiles of the highway scenario as well as the constructed scenarios representing both the morning and afternoon rush periods.}
\footnotesize
\begin{tabular}{lcccccc}
\toprule
Time Interval & \multicolumn{6}{c}{Arrival Rate} \\
\multicolumn{1}{c}{[Time Units]} & \multicolumn{6}{c}{[Number of Users/Time Unit]} \\
\cmidrule(lr{1em}){2-7}
  & Highway Scenario & Scenario 1 & Scenario 2 & Scenario 3 & Scenario 4 & Scenario 5 \\
\midrule
$[0, 14)$ & 1.2 & 2 & 2 & 2 & 2 & 2 \\
$[14, 28)$ & 2 & 2 & 2.5  & 2.25 & 2.25 & 2.5  \\
$[28, 42)$ & 2.25 & 2 & 2 & 2 & 2 & 2 \\
$[42, 56)$ & 2.5 & 2 & 2.5  & 2.25 & 2.5 & 2.25 \\
$[56, \infty)$ & 2.25 & 2 & 2  & 2 & 2 & 2  \\
\bottomrule
\end{tabular} \label{tab:synthetic-scenarios}
\end{table}

\subsection{Results} \label{results-numerical}

\paragraph{Assessment of Theoretical Risk Bounds.} We now assess the theoretical bounds on the risk established in Corollary~\ref{cor:scenario-cor} for both the time-dependent and time-independent data-driven algorithms for the scenarios listed in Table~\ref{tab:synthetic-scenarios}. To this end, we first determined the number of support constraints corresponding to the offline learning problem, i.e., Problem~\eqref{eq:OPT-Obj-offline}-\eqref{eq:OPTcon4-offline} for the time-independent data-driven algorithm and the analogous problem for the time-dependent data-driven algorithm. We then evaluated the solution of the polynomial Equation~\eqref{eq:poly-scenario-otr} for the parameter $\beta = 10^{-6}$, the number of training instances $K = 100$ and the above determined number of support constraints. Having derived the theoretical upper and lower bounds on the risk, we then found the percentage of test instances for which either the competitive ratio constraint or the capacity constraints were violated for each of the scenarios in Table~\ref{tab:synthetic-scenarios}. A total of 100 test instances were considered for each scenario.

Figure~\ref{fig:theory-bounds} depicts for both the time-dependent and time-independent algorithms that the risk, i.e., the probability of constraint violation, on the test instances was within the specified theoretical bounds on the risk as obtained in Corollary~\ref{cor:scenario-cor}. We further note from Figure~\ref{fig:theory-bounds} that both the theoretical bounds on the risk as well as the empirically observed proportion of constraint violation are lower for the time-independent data-driven algorithm. This observation implies a trade-off between the optimal offline objective, which is guaranteed to be lower for the time-dependent data-driven algorithm (see Remark~\ref{rmk:opt-obj-comparison}), and the risk of the optimal parametrization, which is higher for the time-dependent data-driven algorithm.

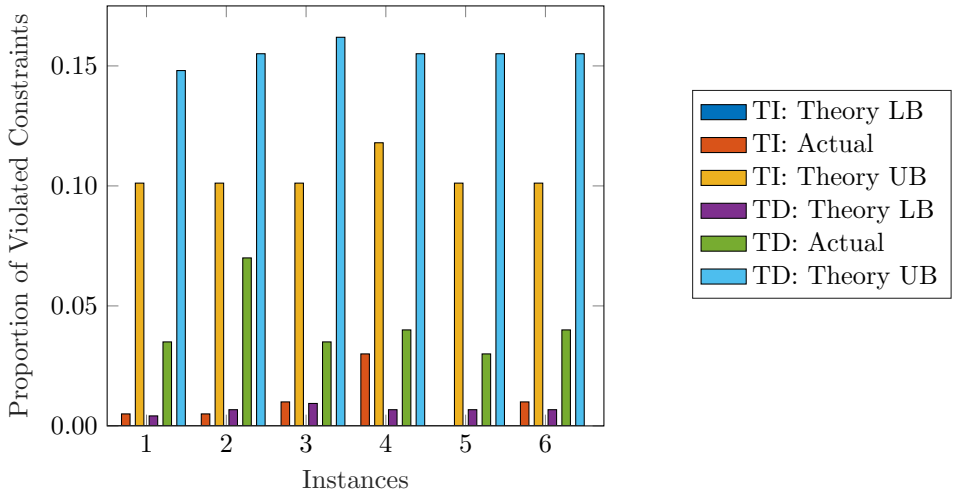
\begin{figure}[tbh!]
  \centering
%
%
\definecolor{mycolor1}{rgb}{0.00000,0.44700,0.74100}%
\definecolor{mycolor2}{rgb}{0.85000,0.32500,0.09800}%
\definecolor{mycolor3}{rgb}{0.92900,0.69400,0.12500}%
\definecolor{mycolor4}{rgb}{0.49400,0.18400,0.55600}%
\definecolor{mycolor5}{rgb}{0.46600,0.67400,0.18800}%
\definecolor{mycolor6}{rgb}{0.30100,0.74500,0.93300}%
\begin{tikzpicture}

\begin{axis}[%
width=2.6in,
height=2.2in,
at={(0in,0in)},
scale only axis,
bar shift auto,
xmin=0.506666666666667,
xmax=6.7333333333333,
xtick={1, 2, 3, 4, 5, 6},
xlabel style={font=\color{white!15!black}},
xlabel={Instances},
ymin=0,
ymax=0.35,
yticklabels = {0, 0.00, 0.05, 0.10, 0.15, 0.20, 0.25, 0.30, 0.35},
ylabel style={font=\color{white!15!black}},
ylabel={Proportion of Violated Constraints},
axis background/.style={fill=white},
legend style={at = {(1.7, 0.8)}, legend cell align=left, align=left, draw=white!15!black}
]
\addplot[ybar, bar width=0.107, fill=mycolor1, draw=black, area legend] table[row sep=crcr] {%
1	0\\
2	0\\
3	0\\
4	0\\
5	0\\
6	0\\
};
\addplot[forget plot, color=white!15!black] table[row sep=crcr] {%
0.506666666666667	0\\
6.49333333333333	0\\
};
\addlegendentry{TI: Theory LB}

\addplot[ybar, bar width=0.107, fill=mycolor2, draw=black, area legend] table[row sep=crcr] {%
1	0.01\\
2	0.01\\
3	0.02\\
4	0.06\\
5	0\\
6	0.02\\
};
\addplot[forget plot, color=white!15!black] table[row sep=crcr] {%
0.506666666666667	0\\
6.49333333333333	0\\
};
\addlegendentry{TI: Actual}

\addplot[ybar, bar width=0.107, fill=mycolor3, draw=black, area legend] table[row sep=crcr] {%
1	0.202324870056299\\
2	0.202324870056299\\
3	0.202324870056299\\
4	0.235962661262276\\
5	0.202324870056299\\
6	0.202324870056299\\
};
\addplot[forget plot, color=white!15!black] table[row sep=crcr] {%
0.506666666666667	0\\
6.49333333333333	0\\
};
\addlegendentry{TI: Theory UB}

\addplot[ybar, bar width=0.107, fill=mycolor4, draw=black, area legend] table[row sep=crcr] {%
1	0.00834051259958124\\
2	0.0135121585749465\\
3	0.018714047187938\\
4	0.0135121585749465\\
5	0.0135121585749465\\
6	0.0135121585749465\\
};
\addplot[forget plot, color=white!15!black] table[row sep=crcr] {%
0.506666666666667	0\\
6.49333333333333	0\\
};
\addlegendentry{TD: Theory LB}

\addplot[ybar, bar width=0.107, fill=mycolor5, draw=black, area legend] table[row sep=crcr] {%
1	0.07\\
2	0.14\\
3	0.07\\
4	0.08\\
5	0.06\\
6	0.08\\
};
\addplot[forget plot, color=white!15!black] table[row sep=crcr] {%
0.506666666666667	0\\
6.49333333333333	0\\
};
\addlegendentry{TD: Actual}

\addplot[ybar, bar width=0.107, fill=mycolor6, draw=black, area legend] table[row sep=crcr] {%
1	0.296129061051403\\
2	0.310155411944916\\
3	0.323866715779936\\
4	0.310155411944916\\
5	0.310155411944916\\
6	0.310155411944916\\
};
\addplot[forget plot, color=white!15!black] table[row sep=crcr] {%
0.506666666666667	0\\
6.49333333333333	0\\
};
\addlegendentry{TD: Theory UB}

\end{axis}

\begin{axis}[%
width=2.8in,
height=2.6in,
at={(0in,0in)},
scale only axis,
xmin=0,
xmax=1,
ymin=0,
ymax=1,
axis line style={draw=none},
ticks=none,
axis x line*=bottom,
axis y line*=left,
legend style={legend cell align=left, align=left, draw=white!15!black}
]
\end{axis}
\end{tikzpicture}%
 \vspace{-10pt}
    \caption{Validation of Theoretical Risk bounds for all instances listed in Table~\ref{tab:synthetic-scenarios}. Here TI and TD represent the results for the time-independent and time-dependent algorithms, respectively. The red and the green bars represent the proportion of instances for which either the competitive ratio constraint or any of the capacity constraints is violated, while the other bars represent the theoretical lower bounds (LB) and upper bounds (UB) on the level of risk as calculated through the number of support constraints for the two algorithms on each of the instances. The dark blue bar representing the theoretical lower bound on the risk for each of the instances and the level of risk on the fifth instance of the time-independent algorithm is infinitesimally close to 0 and so is not visible. } \vspace{-10pt}
    \label{fig:theory-bounds}
\end{figure}

\paragraph{Performance Comparison between Algorithms.} We now investigate the performance of the two data-driven algorithms and that of the greedy algorithm. To provide an equal footing to compare the data-driven and greedy algorithms, we implement discrete allocations corresponding to the learned allocation probabilities from the optimal offline objectives in both the time-dependent and time-independent cases. In particular, for each user we generate a uniform random number between zero and one and assign a quantile range to each arc corresponding to the optimal allocation probabilities based on the value-of-time (and time of arrival) of that user. Then, we implement discrete allocations by assigning each user to the arc for which the uniform random number lies within its quantile range. As an example, suppose we have one user and two arcs, each with allocation probability $0.5$. Then, we can denote the quantile range for the first arc as $[0, 0.5)$ and that of the second arc as $[0.5, 1]$. If the uniform random number is $0.6$, then the user is assigned to the second arc.

To present the results, we denote the time-independent and time-dependent data-driven algorithms, which are based on the scenario approach as ``TI'' and ``TD'' respectively, while we denote the greedy algorithm as ``Greedy''. Figure~\ref{fig:vals_updated} depicts histograms that represent the distribution of total travel cost of the online solution and that of the optimal offline solution on the test instances. All experiments presented in this section were trained on 100 instances and their performance was evaluated on 100 test instances. In the following, we provide a detailed analysis of the results in Figure~\ref{fig:vals_updated}. 

For each of the six scenarios, Figure~\ref{fig:vals_updated} depicts that the ratio between the online and optimal offline solutions of the greedy algorithm has a distribution that is to the right of the corresponding distributions of the data-driven algorithms. Accordingly, both data-driven approaches outperform the greedy algorithm for each of the tested scenarios, since the greedy algorithm tends to achieve generally higher costs. The data-driven algorithms achieve lower total costs since they take into account the impact of routing a given user on the total travel cost. On the other hand, the greedy algorithm only takes local information into account and may allocate some users with the highest value-of-time to arc three, i.e., the highest cost arc, when both the optimal offline solution and the data-driven algorithm allocated those users to lower cost arcs.

\newcommand{\sfwidth}{0.5\columnwidth}
\newcommand{\siwidth}{0.99\linewidth}
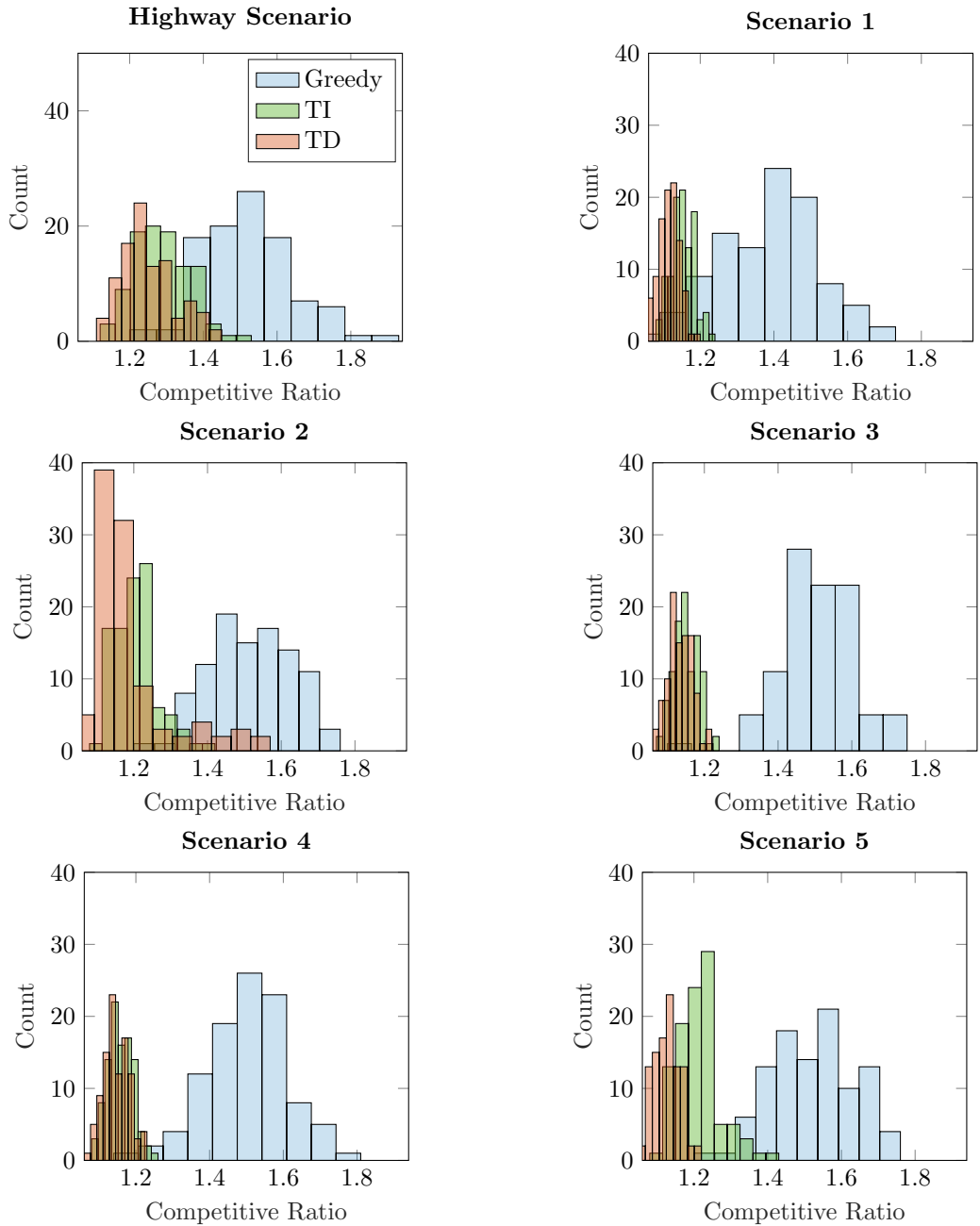
\begin{figure}
    \centering
    \begin{subfigure}[b]{0.48\columnwidth}
        \centering
%
%
\definecolor{mycolor1}{rgb}{0.00000,0.44700,0.74100}%
\definecolor{mycolor2}{rgb}{0.85000,0.32500,0.09800}%
\definecolor{mycolor3}{rgb}{0.3200,0.69400,0.12500}
\begin{tikzpicture}

\begin{axis}[%
width=1.8in,
height=1.6in,
at={(0in,0in)},
scale only axis,
xmin=1.06,
xmax=1.94,
xlabel style={font=\color{white!15!black}},
xlabel={Competitive Ratio},
ymin=0,
ymax=50,
ylabel style={font=\color{white!15!black}},
ylabel={Count},
axis background/.style={fill=white},
title style={font=\bfseries},
title={Highway Scenario},
legend style={legend cell align=left, align=left, draw=white!15!black}
]
\addplot[ybar interval, fill=mycolor1, fill opacity=0.2, draw=black, area legend] table[row sep=crcr] {%
x	y\\
1.2	2\\
1.273	2\\
1.346	18\\
1.419	20\\
1.492	26\\
1.565	18\\
1.638	7\\
1.711	6\\
1.784	1\\
1.857	1\\
1.93	1\\
};
\addlegendentry{Greedy}

\addplot[ybar interval, fill=mycolor3, fill opacity=0.4, draw=black, area legend] table[row sep=crcr] {%
x	y\\
1.12	3\\
1.161	9\\
1.202	19\\
1.243	20\\
1.284	19\\
1.325	13\\
1.366	13\\
1.407	3\\
1.448	1\\
1.489	1\\
1.53	1\\
};
\addlegendentry{TI}

\addplot[ybar interval, fill=mycolor2, fill opacity=0.4, draw=black, area legend] table[row sep=crcr] {%
x	y\\
1.11	4\\
1.144	11\\
1.178	17\\
1.212	24\\
1.246	13\\
1.28	14\\
1.314	4\\
1.348	7\\
1.382	5\\
1.416	2\\
1.45	2\\
};
\addlegendentry{TD}

\end{axis}
\end{tikzpicture}%
    \end{subfigure}
    \begin{subfigure}[b]{0.48\columnwidth}
        \centering
%
%
\definecolor{mycolor1}{rgb}{0.00000,0.44700,0.74100}%
\definecolor{mycolor2}{rgb}{0.85000,0.32500,0.09800}%
\definecolor{mycolor3}{rgb}{0.3200,0.69400,0.12500}
\begin{tikzpicture}

\begin{axis}[%
width=1.8in,
height=1.6in,
at={(0in,0in)},
scale only axis,
xmin=1.06,
xmax=1.94,
xlabel style={font=\color{white!15!black}},
xlabel={Competitive Ratio},
ymin=0,
ymax=40,
ylabel style={font=\color{white!15!black}},
ylabel={Count},
title style={font=\bfseries},
title={Scenario 1},
axis background/.style={fill=white},
title style={font=\bfseries},
legend style={legend cell align=left, align=left, draw=white!15!black}
]
\addplot[ybar interval, fill=mycolor1, fill opacity=0.2, draw=black, area legend] table[row sep=crcr] {%
x	y\\
1.02	1\\
1.091	4\\
1.162	9\\
1.233	15\\
1.304	13\\
1.375	24\\
1.446	20\\
1.517	8\\
1.588	5\\
1.659	2\\
1.73	2\\
};

\addplot[ybar interval, fill=mycolor3, fill opacity=0.4, draw=black, area legend] table[row sep=crcr] {%
x	y\\
1.08	3\\
1.096	9\\
1.112	9\\
1.128	20\\
1.144	21\\
1.16	13\\
1.176	18\\
1.192	3\\
1.208	4\\
1.224	1\\
1.24	1\\
};

\addplot[ybar interval, fill=mycolor2, fill opacity=0.4, draw=black, area legend] table[row sep=crcr] {%
x	y\\
1.04	3\\
1.056	6\\
1.072	9\\
1.088	17\\
1.104	21\\
1.12	22\\
1.136	14\\
1.152	7\\
1.168	1\\
1.184	1\\
1.2	1\\
};

\end{axis}
\end{tikzpicture}%
    \end{subfigure}\par\medskip
    \vspace{-20pt}
    \hspace{0pt}
    \begin{subfigure}[b]{0.48\columnwidth}
        \centering
%
%
\definecolor{mycolor1}{rgb}{0.00000,0.44700,0.74100}%
\definecolor{mycolor2}{rgb}{0.85000,0.32500,0.09800}%
\definecolor{mycolor3}{rgb}{0.3200,0.69400,0.12500}
\begin{tikzpicture}

\begin{axis}[%
width=1.8in,
height=1.6in,
at={(0in,0in)},
scale only axis,
xmin=1.06,
xmax=1.94,
xlabel style={font=\color{white!15!black}},
xlabel={Competitive Ratio},
ymin=0,
ymax=40,
ylabel style={font=\color{white!15!black}},
ylabel={Count},
title style={font=\bfseries},
title={Scenario 2},
axis background/.style={fill=white},
title style={font=\bfseries},
legend style={legend cell align=left, align=left, draw=white!15!black}
]
\addplot[ybar interval, fill=mycolor1, fill opacity=0.2, draw=black, area legend] table[row sep=crcr] {%
x	y\\
1.2	1\\
1.256	1\\
1.312	8\\
1.368	12\\
1.424	19\\
1.48	15\\
1.536	17\\
1.592	14\\
1.648	11\\
1.704	3\\
1.76	3\\
};

\addplot[ybar interval, fill=mycolor3, fill opacity=0.4, draw=black, area legend] table[row sep=crcr] {%
x	y\\
1.08	1\\
1.114	17\\
1.148	17\\
1.182	24\\
1.216	26\\
1.25	6\\
1.284	5\\
1.318	3\\
1.352	1\\
1.386	1\\
1.42	1\\
};

\addplot[ybar interval, fill=mycolor2, fill opacity=0.4, draw=black, area legend] table[row sep=crcr] {%
x	y\\
1.04	5\\
1.093	39\\
1.146	32\\
1.199	9\\
1.252	3\\
1.305	2\\
1.358	4\\
1.411	2\\
1.464	3\\
1.517	2\\
1.57	2\\
};

\end{axis}
\end{tikzpicture}%
    \end{subfigure}
    \begin{subfigure}[b]{0.48\columnwidth}
        \centering
%
%
\definecolor{mycolor1}{rgb}{0.00000,0.44700,0.74100}%
\definecolor{mycolor2}{rgb}{0.85000,0.32500,0.09800}%
\definecolor{mycolor3}{rgb}{0.3200,0.69400,0.12500}
\begin{tikzpicture}

\begin{axis}[%
width=1.8in,
height=1.6in,
at={(0in,0in)},
scale only axis,
xmin=1.06,
xmax=1.94,
xlabel style={font=\color{white!15!black}},
xlabel={Competitive Ratio},
ymin=0,
ymax=40,
ylabel style={font=\color{white!15!black}},
ylabel={Count},
title style={font=\bfseries},
title={Scenario 3},
axis background/.style={fill=white},
title style={font=\bfseries},
legend style={legend cell align=left, align=left, draw=white!15!black}
]
\addplot[ybar interval, fill=mycolor1, fill opacity=0.2, draw=black, area legend] table[row sep=crcr] {%
x	y\\
1.1	1\\
1.165	0\\
1.23	0\\
1.295	5\\
1.36	11\\
1.425	28\\
1.49	23\\
1.555	23\\
1.62	5\\
1.685	5\\
1.75	5\\
};

\addplot[ybar interval, fill=mycolor3, fill opacity=0.4, draw=black, area legend] table[row sep=crcr] {%
x	y\\
1.07	2\\
1.087	7\\
1.104	11\\
1.121	18\\
1.138	22\\
1.155	11\\
1.172	16\\
1.189	11\\
1.206	1\\
1.223	2\\
1.24	2\\
};

\addplot[ybar interval, fill=mycolor2, fill opacity=0.4, draw=black, area legend] table[row sep=crcr] {%
x	y\\
1.06	3\\
1.076	7\\
1.092	10\\
1.108	22\\
1.124	15\\
1.14	16\\
1.156	16\\
1.172	8\\
1.188	1\\
1.204	3\\
1.22	3\\
};

\end{axis}
\end{tikzpicture}%
    \end{subfigure}\par\medskip 
    \vspace{-20pt}
    \hspace{0pt}
    \begin{subfigure}[b]{0.48\columnwidth}
        \centering
%
%
\definecolor{mycolor1}{rgb}{0.00000,0.44700,0.74100}%
\definecolor{mycolor2}{rgb}{0.85000,0.32500,0.09800}%
\definecolor{mycolor3}{rgb}{0.3200,0.69400,0.12500}
\begin{tikzpicture}

\begin{axis}[%
width=1.8in,
height=1.6in,
at={(0in,0in)},
scale only axis,
xmin=1.06,
xmax=1.94,
xlabel style={font=\color{white!15!black}},
xlabel={Competitive Ratio},
ymin=0,
ymax=40,
ylabel style={font=\color{white!15!black}},
ylabel={Count},
title style={font=\bfseries},
title={Scenario 4},
axis background/.style={fill=white},
title style={font=\bfseries},
legend style={legend cell align=left, align=left, draw=white!15!black}
]
\addplot[ybar interval, fill=mycolor1, fill opacity=0.2, draw=black, area legend] table[row sep=crcr] {%
x	y\\
1.14	1\\
1.207	2\\
1.274	4\\
1.341	12\\
1.408	19\\
1.475	26\\
1.542	23\\
1.609	8\\
1.676	5\\
1.743	1\\
1.81	1\\
};

\addplot[ybar interval, fill=mycolor3, fill opacity=0.4, draw=black, area legend] table[row sep=crcr] {%
x	y\\
1.08	3\\
1.098	8\\
1.116	14\\
1.134	22\\
1.152	16\\
1.17	17\\
1.188	14\\
1.206	4\\
1.224	2\\
1.242	1\\
1.26	1\\
};

\addplot[ybar interval, fill=mycolor2, fill opacity=0.4, draw=black, area legend] table[row sep=crcr] {%
x	y\\
1.06	1\\
1.077	5\\
1.094	9\\
1.111	15\\
1.128	23\\
1.145	12\\
1.162	17\\
1.179	12\\
1.196	3\\
1.213	4\\
1.23	4\\
};

\end{axis}
\end{tikzpicture}%
    \end{subfigure}\hspace{-5pt}
    \begin{subfigure}[b]{0.48\columnwidth}
        \centering
%
%
\definecolor{mycolor1}{rgb}{0.00000,0.44700,0.74100}%
\definecolor{mycolor2}{rgb}{0.85000,0.32500,0.09800}%
\definecolor{mycolor3}{rgb}{0.3200,0.69400,0.12500}
\begin{tikzpicture}

\begin{axis}[%
width=1.8in,
height=1.6in,
at={(0in,0in)},
scale only axis,
xmin=1.06,
xmax=1.94,
xlabel style={font=\color{white!15!black}},
xlabel={Competitive Ratio},
ymin=0,
ymax=40,
ylabel style={font=\color{white!15!black}},
ylabel={Count},
title style={font=\bfseries},
title={Scenario 5},
axis background/.style={fill=white},
title style={font=\bfseries},
legend style={legend cell align=left, align=left, draw=white!15!black}
]
\addplot[ybar interval, fill=mycolor1, fill opacity=0.2, draw=black, area legend] table[row sep=crcr] {%
x	y\\
1.2	1\\
1.256	1\\
1.312	6\\
1.368	13\\
1.424	18\\
1.48	14\\
1.536	21\\
1.592	10\\
1.648	13\\
1.704	4\\
1.76	4\\
};

\addplot[ybar interval, fill=mycolor3, fill opacity=0.4, draw=black, area legend] table[row sep=crcr] {%
x	y\\
1.08	1\\
1.115	13\\
1.15	19\\
1.185	24\\
1.22	29\\
1.255	5\\
1.29	5\\
1.325	3\\
1.36	1\\
1.395	1\\
1.43	1\\
};

\addplot[ybar interval, fill=mycolor2, fill opacity=0.4, draw=black, area legend] table[row sep=crcr] {%
x	y\\
1.03	1\\
1.049	2\\
1.068	13\\
1.087	15\\
1.106	17\\
1.125	23\\
1.144	13\\
1.163	13\\
1.182	2\\
1.201	2\\
1.22	2\\
};

\end{axis}
\end{tikzpicture}%
    \end{subfigure}
    \vspace{-15pt}
    \caption{{Distribution of the ratio of the total travel cost of the online solution of the (i) greedy, (ii) time-independent (TI), and (iii) time-dependent (TD) algorithms to that of the optimal offline solution for the six scenarios on the test instances. Here competitive ratio is used to represent the ratio of the total travel cost between the online and offline solutions for each instance.}}
    \label{fig:vals_updated}
\end{figure}


Next, we compare the performance of the two data-driven algorithms on the test instances. To this end, first note from Figure~\ref{fig:vals_updated} that the distribution of the ratio between the online and optimal offline solutions of the time-independent data driven algorithm is generally to the right of that of the time-dependent data-driven algorithm. This indicates the superior performance of the time-dependent data-driven algorithm, which can be expected since it allocated users with different probabilities as the user arrival rate changes, which is not the case with the time-independent algorithm. The slightly higher ratio of the right tail of the distribution between the online and optimal offline solutions of the time-dependent algorithm in Scenario 2 (cf. Figure~\ref{fig:vals_updated}) is attributable to the higher level of risk of the time-dependent algorithm, as was observed in the earlier theoretical validation of the risk bounds. However, other than in Scenario 2 even the right tail of the distribution of the ratio between the online and optimal offline solutions of the time-dependent data driven algorithm is lower than that of the time-independent algorithm. The aforementioned observation holds despite the fact that the generalization guarantee for the time-dependent algorithm is worse than that of the time-independent algorithm, since the risk of the former is lower than that of the latter on the test instances. 


To summarize, Figure~\ref{fig:vals_updated} indicates that the time-dependent algorithm achieves the best performance while the greedy algorithm achieves the highest ratios between the online and optimal offline solutions, thereby highlighting the value of prior information in online algorithm design. Furthermore, the results indicate that despite a superior probabilistic generalization guarantee, i.e., a lower probability of constraint violation on unseen sequences, for the time-independent algorithm, the right tail of the distribution of the ratios between the online and optimal offline solutions of the time-dependent algorithm is either slightly higher or strictly lower than that of the time-independent algorithm. The superior numerical performance of the time-dependent algorithm despite the better generalization guarantee of the time-independent one thus points towards the importance of the ability to adjust allocations based on changes in user arrival rates.

\section{Conclusions and Future Work} \label{conclusions}
This paper studied online traffic routing from a competitive analysis as well as a data-driven perspective. We studied both the adversarial and stochastic user arrival settings for the \acrshort{otr} problem. For the adversarial model of user arrivals, we showed that the greedy algorithm achieves a competitive ratio of one for the two-arc \acrshort{otr-i} problem. We then established that in extended variants of the two-arc \acrshort{otr-i} problem, the worst-case ratio between the total travel cost of the online solution of any feasible deterministic algorithm and that of the optimal offline solution is unbounded. In the stochastic setting, we presented both time-independent and time-dependent data-driven algorithms, which leverage the scenario approach, and established bounds on the the probability of constraint violation of the optimal parametrizations computed through these algorithms on unseen test instances. Finally, we presented numerical experiments based on an application case of the San Francisco Bay Area and compared the data-driven and greedy algorithms' performance. In particular, our results indicate that both the time-independent and time-dependent data-driven algorithms outperform the greedy algorithm, highlighting the value of prior information in online algorithm design. Furthermore, we observe that the time-dependent algorithm generally has better performance as compared to the time-independent one despite a superior probabilistic generalization guarantee of the time-independent algorithm.



There are various directions for future research. 
First, it would be worthwhile to investigate the tractability and generalizability of the data-driven approach in road networks beyond the parallel network setting considered in this work. Furthermore, it would be interesting to investigate any theoretical underpinnings that explain the superior numerical performance of the time-dependent algorithm despite a superior probabilistic generalization guarantee of the time-independent algorithm.
Finally, it would be worthwhile to extend the probabilistic generalization guarantee of the data-driven algorithm that makes fractional allocations to the setting of discrete allocations.


\section*{Acknowledgements}
This work was supported by the National Science Foundation (NSF) Award 1830554.


\bibliographystyle{unsrtnat}
\bibliography{main}

\begin{thebibliography}{39}
\providecommand{\natexlab}[1]{#1}
\providecommand{\url}[1]{\texttt{#1}}
\expandafter\ifx\csname urlstyle\endcsname\relax
  \providecommand{\doi}[1]{doi: #1}\else
  \providecommand{\doi}{doi: \begingroup \urlstyle{rm}\Url}\fi

\bibitem[Analytics(2018)]{maps-users-2018}
Verto Analytics.
\newblock Most popular mapping apps in the united states as of april 2018, by
  monthly users (in millions), 2018.
\newblock URL
  \url{https://www.statista.com/statistics/865413/most-popular-us-mapping-apps-ranked-by-audience/}.
\newblock {Accessed} February 04, 2021.

\bibitem[Cabannes et~al.(2018)Cabannes, Fighiera, Ugirumurera, Sundt, and
  Bayen]{Cabannes2017TheIO}
Theophile Cabannes, Vincent Fighiera, Juliette Ugirumurera, Alexander Sundt,
  and Alexandre Bayen.
\newblock The impact of gps-enabled shortest path routing on mobility: a game
  theoretic approach.
\newblock 2018.

\bibitem[Hart et~al.(1968)Hart, Nilsson, and Raphael]{Hart1968}
Peter~E. Hart, Nils~J. Nilsson, and Bertram Raphael.
\newblock A formal basis for the heuristic determination of minimum cost paths.
\newblock \emph{IEEE Transactions on Systems Science and Cybernetics},
  4\penalty0 (2):\penalty0 100--107, July 1968.
\newblock ISSN 0536-1567.
\newblock \doi{10.1109/TSSC.1968.300136}.

\bibitem[Patriksson(2015)]{TAP}
Michael Patriksson.
\newblock \emph{The Traffic Assignment Problem: Models and Methods}.
\newblock 02 2015.
\newblock ISBN 978-0486787909.

\bibitem[Nash(1950)]{Nash48}
John~F. Nash.
\newblock Equilibrium points in n-person games.
\newblock \emph{Proceedings of the National Academy of Sciences}, 36\penalty0
  (1):\penalty0 48--49, 1950.
\newblock ISSN 0027-8424.
\newblock \doi{10.1073/pnas.36.1.48}.
\newblock URL \url{https://www.pnas.org/content/36/1/48}.

\bibitem[Bilali et~al.(2019)Bilali, Isaac, Amini, and
  Motamedidehkordi]{BILALI2019494}
Aledia Bilali, Gordon Isaac, Sasan Amini, and Nassim Motamedidehkordi.
\newblock Analyzing the impact of anticipatory vehicle routing on the network
  performance.
\newblock \emph{Transportation Research Procedia}, 41:\penalty0 494--506, 2019.
\newblock ISSN 2352-1465.
\newblock \doi{https://doi.org/10.1016/j.trpro.2019.09.082}.
\newblock URL
  \url{https://www.sciencedirect.com/science/article/pii/S2352146519304995}.

\bibitem[Dong et~al.(2006)Dong, Mahmassani, and
  Lu]{anticipatory-route-guidance}
Jing Dong, Hani~S. Mahmassani, and Chung-Cheng Lu.
\newblock How reliable is this route?: Predictive travel time and reliability
  for anticipatory traveler information systems.
\newblock \emph{Transportation Research Record}, 1980\penalty0 (1):\penalty0
  117--125, 2006.
\newblock \doi{10.1177/0361198106198000116}.
\newblock URL \url{https://doi.org/10.1177/0361198106198000116}.

\bibitem[Hwang et~al.(2018)Hwang, Jaillet, and Manshadi]{Hwang2018OnlineRA}
Dawsen Hwang, Patrick Jaillet, and Vahideh Manshadi.
\newblock Online resource allocation under partially predictable demand.
\newblock \emph{Computation Theory eJournal}, 2018.

\bibitem[Pigou(1920)]{pigou}
Arthur Pigou.
\newblock \emph{The Economics of Welfare}.
\newblock Macmillan, 1920.

\bibitem[Roughgarden and Tardos(2002)]{selfish-routing}
Tim Roughgarden and Eva Tardos.
\newblock How bad is selfish routing?
\newblock \emph{Journal of the ACM}, 49:\penalty0 236--259, 2002.

\bibitem[Krichene et~al.(2017)Krichene, Reilly, Amin, and
  Bayen]{amin-stackelberg}
Walid Krichene, Jack~D. Reilly, Saurabh Amin, and Alexandre~M. Bayen.
\newblock \emph{Stackelberg Routing on Parallel Transportation Networks}, pages
  1--35.
\newblock Springer International Publishing, Cham, 2017.
\newblock ISBN 978-3-319-27335-8.
\newblock \doi{10.1007/978-3-319-27335-8_26-1}.
\newblock URL \url{https://doi.org/10.1007/978-3-319-27335-8_26-1}.

\bibitem[Zhou et~al.(2008)Zhou, Chakrabarty, and
  Lukose]{online-knapsack-chakrabarty}
Yunhong Zhou, Deeparnab Chakrabarty, and Rajan Lukose.
\newblock Budget constrained bidding in keyword auctions and online knapsack
  problems.
\newblock In Christos Papadimitriou and Shuzhong Zhang, editors, \emph{Internet
  and Network Economics}, pages 566--576, Berlin, Heidelberg, 2008. Springer
  Berlin Heidelberg.
\newblock ISBN 978-3-540-92185-1.

\bibitem[Potts and Strusevich(2009)]{jobScheduling}
C.~N. Potts and V.~A. Strusevich.
\newblock Fifty years of scheduling: A survey of milestones.
\newblock \emph{The Journal of the Operational Research Society}, 60:\penalty0
  s41--s68, 2009.
\newblock ISSN 01605682, 14769360.
\newblock URL \url{http://www.jstor.org/stable/40206725}.

\bibitem[Caltrans(2010)]{caltrans}
Caltrans.
\newblock Us 101 south, corridor system management plan, 2010, 2010.
\newblock Accessed July 01, 2020.

\bibitem[Beesley(1965)]{vot-learn}
M.~E. Beesley.
\newblock The value of time spent in travelling: Some new evidence.
\newblock \emph{Economica}, 32\penalty0 (126):\penalty0 174--185, 1965.
\newblock ISSN 00130427, 14680335.
\newblock URL \url{http://www.jstor.org/stable/2552547}.

\bibitem[Cole et~al.(2003)Cole, Dodis, and
  Roughgarden]{heterogeneous-pricing-roughgarden}
Richard Cole, Yevgeniy Dodis, and Tim Roughgarden.
\newblock Pricing network edges for heterogeneous selfish users.
\newblock In \emph{Symposium on Theory of Computing}, page 521–530.
  Association for Computing Machinery, 2003.

\bibitem[{Fleischer} et~al.(2004){Fleischer}, {Jain}, and
  {Mahdian}]{multicommodity-extension}
L.~{Fleischer}, K.~{Jain}, and M.~{Mahdian}.
\newblock Tolls for heterogeneous selfish users in multicommodity networks and
  generalized congestion games.
\newblock In \emph{Symposium on Foundations of Computer Science}, pages
  277--285. IEEE, 2004.

\bibitem[Yang and Huang(2005)]{METRP}
H.~Yang and H.~J. Huang.
\newblock \emph{Mathematical and Economic Theory of Road Pricing}.
\newblock Emerald Publishing, 1 edition, 2005.
\newblock ISBN 978-0486787909.

\bibitem[Sun et~al.(2020)Sun, Zeynali, Li, Hajiesmaili, Wierman, and
  Tsang]{online-knapsack-wierman}
Bo~Sun, Ali Zeynali, Tongxin Li, Mohammad Hajiesmaili, Adam Wierman, and
  Danny~H.K. Tsang.
\newblock Competitive algorithms for the online multiple knapsack problem with
  application to electric vehicle charging.
\newblock \emph{Proc. ACM Meas. Anal. Comput. Syst.}, 4\penalty0 (3), November
  2020.
\newblock \doi{10.1145/3428336}.
\newblock URL \url{https://doi.org/10.1145/3428336}.

\bibitem[Pruhs et~al.(2003)Pruhs, Sgall, and Torng]{Pruhs03onlinescheduling}
Kirk Pruhs, Jiri Sgall, and Eric Torng.
\newblock Online scheduling.
\newblock pages 115--124. CRC Press, 2003.

\bibitem[Jaillet and Wagner(2008)]{online-TSP-jaillet}
Patrick Jaillet and Michael~R. Wagner.
\newblock Generalized online routing: New competitive ratios, resource
  augmentation, and asymptotic analyses.
\newblock \emph{Operations Research}, 56\penalty0 (3):\penalty0 745--757, 2008.
\newblock \doi{10.1287/opre.1070.0450}.
\newblock URL \url{https://doi.org/10.1287/opre.1070.0450}.

\bibitem[Mehta et~al.(2007)Mehta, Saberi, Vazirani, and Vazirani]{msvv}
Aranyak Mehta, Amin Saberi, Umesh Vazirani, and Vijay Vazirani.
\newblock Adwords and generalized online matching.
\newblock \emph{J. ACM}, 54\penalty0 (5):\penalty0 22–es, October 2007.
\newblock ISSN 0004-5411.
\newblock \doi{10.1145/1284320.1284321}.
\newblock URL \url{https://doi.org/10.1145/1284320.1284321}.

\bibitem[Buchbinder and Naor(2009)]{primal-dual-buchbinder}
Niv Buchbinder and Joseph Naor.
\newblock The design of competitive online algorithms via a primal-dual
  approach.
\newblock \emph{Found. Trends Theor. Comput. Sci.}, 3\penalty0
  (2–3):\penalty0 93–263, February 2009.
\newblock ISSN 1551-305X.
\newblock \doi{10.1561/0400000024}.
\newblock URL \url{https://doi.org/10.1561/0400000024}.

\bibitem[Devanur et~al.(2011)Devanur, Jain, Sivan, and
  Wilkens]{devanur-online-resource}
Nikhil~R. Devanur, Kamal Jain, Balasubramanian Sivan, and Christopher~A.
  Wilkens.
\newblock Near optimal online algorithms and fast approximation algorithms for
  resource allocation problems.
\newblock In \emph{Proceedings of the 12th ACM Conference on Electronic
  Commerce}, EC '11, page 29–38, New York, NY, USA, 2011. Association for
  Computing Machinery.
\newblock ISBN 9781450302616.
\newblock \doi{10.1145/1993574.1993581}.
\newblock URL \url{https://doi.org/10.1145/1993574.1993581}.

\bibitem[Motwani et~al.(1998)Motwani, Saraswat, and Torng]{motwani-scheduling}
Rajeev Motwani, Vijay Saraswat, and Eric Torng.
\newblock Online scheduling with lookahead: Multipass assembly lines.
\newblock \emph{INFORMS Journal on Computing}, 10\penalty0 (3):\penalty0
  331--340, 1998.
\newblock \doi{10.1287/ijoc.10.3.331}.
\newblock URL \url{https://doi.org/10.1287/ijoc.10.3.331}.

\bibitem[Koutsoupias and Papadimitriou(2000)]{beyond-comp-analysis}
Elias Koutsoupias and Christos~H. Papadimitriou.
\newblock Beyond competitive analysis.
\newblock \emph{SIAM J. Comput.}, 30\penalty0 (1):\penalty0 300–317, April
  2000.
\newblock ISSN 0097-5397.
\newblock \doi{10.1137/S0097539796299540}.
\newblock URL \url{https://doi.org/10.1137/S0097539796299540}.

\bibitem[Roughgarden(2018)]{roughgarden2018worstcase}
Tim Roughgarden.
\newblock Beyond worst-case analysis, 2018.

\bibitem[Blom et~al.(2001)Blom, Krumke, de~Paepe, and Stougie]{online-tsp-fair}
Michiel Blom, Sven~O. Krumke, Willem~E. de~Paepe, and Leen Stougie.
\newblock The online tsp against fair adversaries.
\newblock \emph{INFORMS Journal on Computing}, 13\penalty0 (2):\penalty0
  138--148, 2001.
\newblock \doi{10.1287/ijoc.13.2.138.10517}.
\newblock URL \url{https://doi.org/10.1287/ijoc.13.2.138.10517}.

\bibitem[Mehta(2013)]{TCS-057}
Aranyak Mehta.
\newblock Online matching and ad allocation.
\newblock \emph{Foundations and Trends in Theoretical Computer Science},
  8\penalty0 (4):\penalty0 265--368, 2013.
\newblock ISSN 1551-305X.
\newblock \doi{10.1561/0400000057}.
\newblock URL \url{http://dx.doi.org/10.1561/0400000057}.

\bibitem[{Feldman} et~al.(2009){Feldman}, {Mehta}, {Mirrokni}, and
  {Muthukrishnan}]{stochastic-matching-beating}
Jon {Feldman}, Aranyak {Mehta}, Vahab {Mirrokni}, and S.~{Muthukrishnan}.
\newblock Online stochastic matching: Beating 1-1/e.
\newblock In \emph{2009 50th Annual IEEE Symposium on Foundations of Computer
  Science}, pages 117--126, 2009.
\newblock \doi{10.1109/FOCS.2009.72}.

\bibitem[Devanur and Hayes(2009)]{devanur-adwords}
Nikhil~R. Devanur and Thomas~P. Hayes.
\newblock The adwords problem: Online keyword matching with budgeted bidders
  under random permutations.
\newblock In \emph{Proceedings of the 10th ACM Conference on Electronic
  Commerce}, EC '09, page 71–78, New York, NY, USA, 2009. Association for
  Computing Machinery.
\newblock ISBN 9781605584584.
\newblock \doi{10.1145/1566374.1566384}.
\newblock URL \url{https://doi.org/10.1145/1566374.1566384}.

\bibitem[Agrawal et~al.(2014)Agrawal, Wang, and Ye]{online-agrawal}
Shipra Agrawal, Zizhuo Wang, and Yinyu Ye.
\newblock A dynamic near-optimal algorithm for online linear programming.
\newblock \emph{Operations Research}, 62\penalty0 (4):\penalty0 876--890, 2014.

\bibitem[{Calafiore} and {Campi}(2006)]{scenario-main}
Giuseppe~C. {Calafiore} and Marco~C. {Campi}.
\newblock The scenario approach to robust control design.
\newblock \emph{IEEE Transactions on Automatic Control}, 51\penalty0
  (5):\penalty0 742--753, 2006.
\newblock \doi{10.1109/TAC.2006.875041}.

\bibitem[Campi and Garatti(2008)]{campi2008exact}
Marco~C Campi and Simone Garatti.
\newblock The exact feasibility of randomized solutions of uncertain convex
  programs.
\newblock \emph{SIAM Journal on Optimization}, 19\penalty0 (3):\penalty0
  1211--1230, 2008.

\bibitem[Garatti and Campi(2019)]{risk-complexity-scenario}
Simone Garatti and Marco Campi.
\newblock Risk and complexity in scenario optimization.
\newblock \emph{Mathematical Programming}, 11 2019.
\newblock \doi{10.1007/s10107-019-01446-4}.

\bibitem[TomTom(2021)]{tom-tom-data}
TomTom.
\newblock Tomtom developer portal, 2021.
\newblock URL \url{https://developer.tomtom.com/}.
\newblock {Accessed} January 25, 2021.

\bibitem[SFU(2018)]{SFUrgency}
Household income in san francisco, california, 2018.
\newblock URL
  \url{https://statisticalatlas.com/place/California/San-Francisco/Household-Income#figure/household-income-percentiles}.
\newblock {Accessed} January 20, 2021.

\bibitem[Thomas and Thompson(1970)]{thomas1970value}
Thomas~C Thomas and Gordon~I Thompson.
\newblock The value of time for commuting motorists as a function of their
  income level and amount of time saved.
\newblock \emph{Highway Research Record}, \penalty0 (314), 1970.

\bibitem[Chen and Varaiya(2002)]{pems-database}
Chao Chen and Pravin Varaiya.
\newblock \emph{Freeway Performance Measurement System (Pems)}.
\newblock PhD thesis, 2002.
\newblock AAI3082138.

\end{thebibliography}


%
%
%


\appendix
\section{Proofs}
For notational convenience, we drop the superscript $\mathcal{I}$ in $N_{a, \mathbf{A}}^{\mathcal{I}}$, the number of users in instance $\mathcal{I}$ that are allocated on arc $a$ by Algorithm $\mathbf{A}$, and note that the proofs presented in this section hold for all instances $\mathcal{I} \in \Omega$.

\subsection{Proof of Lemma~\ref{lem:lemm1}} \label{proof-lem-1}

We can evaluate the total cost $C_{\mathbf{A}}$ of algorithm $\mathbf{A}$ as:

\begin{align}
    C_{\mathbf{A}} = \sum_{a = 1}^{M} t_a N_{a, \mathbf{A}}
\end{align}
Note that this is equivalent to the objective function~\eqref{eq:OPT-integer-Obj}, with $\theta_i = 1$ for all agents~$i$. 

We now consider a \textit{feasible} algorithm $\mathbf{A}'$ that satisfies Constraints~\eqref{eq:prob-constraint}-\eqref{eq:capacity-constraints}. We first note that the number of people that must be allocated under $\mathbf{A}$ and $\mathbf{A}'$ must be equal: $\sum_{a = 1}^{M} N_{a, \mathbf{A}} = \sum_{a = 1}^{M} N_{a, \mathbf{A}'}$. To prove the lemma we first consider a simplified case with two arcs, which we then extend to the more general case with $M$ arcs.

\paragraph{Two arc Case:} For algorithm $\mathbf{A}$ we denote the total number of people allocated to arc 1 as $N_{1, \mathbf{A}}$ and the number allocated to arc 2 as $N_{2, \mathbf{A}}$. Then the total number of people in the system is $N = N_{1, \mathbf{A}}+N_{2, \mathbf{A}}$. The total cost of algorithm $\mathbf{A}$ can then be written as:
\[ C_{\mathbf{A}} = t_1 N_{1, \mathbf{A}} + t_2 N_{2, \mathbf{A}} = t_1 N_{1, \mathbf{A}'} + t_1(N_{1, \mathbf{A}}-N_{1, \mathbf{A}'}) + t_2 N_{2, \mathbf{A}}\]
Using the property of algorithm $\mathbf{A}$ that $N_{1, \mathbf{A}} \geq N_{1, \Bar{\mathbf{A}}}$ for any feasible $\Bar{\mathbf{A}}$ it holds for $\mathbf{A}'$ that $N_{1, \mathbf{A}} - N_{1, \mathbf{A}'} \geq 0$, and
\[ C_\mathbf{A} \leq t_1 N_{1, \mathbf{A}'} + t_2(N_{2, \mathbf{A}} + N_{1, \mathbf{A}}-N_{1, \mathbf{A}'})\]
follows. Since both $\mathbf{A}$ and $\mathbf{A}'$ allocate the same number of users, we have $N_{1, \mathbf{A}}+N_{2, \mathbf{A}} = N_{1, \mathbf{A}'}+N_{2, \mathbf{A}'}$, which implies that:
\[  C_\mathbf{A} \leq t_1 N_{1, \mathbf{A}'} + t_2 N_{2, \mathbf{A}'} = C_{\mathbf{A}'}\]
Thus, we have established for the two-arc case that $C_\mathbf{A} \leq C_{\mathbf{A}'}$ for all \textit{feasible} algorithms $\textbf{A}'$.

\paragraph{Multiple arc Case:} We now generalize this proof strategy to a parallel road network with $M$ arcs. We first observe from the definition of an algorithm's total cost that:
\[ C_\mathbf{A} = t_1 N_{1, \mathbf{A}} + t_2 N_{2, \mathbf{A}} + ... + t_M N_{M, \mathbf{A}}\]
Next, for any \textit{feasible} algorithm $\mathbf{A}'$ we can rewrite this expression as:
\[ C_\mathbf{A} = t_1 N_{1, \mathbf{A}'} + t_1(N_{1, \mathbf{A}}-N_{1, \mathbf{A}'}) + t_2 N_{2, \mathbf{A}'} + t_2 (N_{2, \mathbf{A}} - N_{2, \mathbf{A}}) + t_3 N_{3, \mathbf{A}} + ... + t_M N_{M, \mathbf{A}}\]
Then, observing that $N_{1, \mathbf{A}} \geq N_{1, \mathbf{A}'}$, we obtain:
\[ C_\mathbf{A} \leq t_1 N_{1, \mathbf{A}'} + t_2 N_{2, \mathbf{A}'} + t_2(N_{1, \mathbf{A}}-N_{1, \mathbf{A}'} + N_{2, \mathbf{A}} - N_{2, \mathbf{A}}) + t_3 N_{3, \mathbf{A}} ... + t_M N_{M, \mathbf{A}'}\]
Now, we can repeat this process and use that $\sum_{a = 1}^{k} N_{a, \mathbf{A}} \geq \sum_{a = 1}^{k} N_{a, \mathbf{A}'}$ $\forall k \in [M]$ to get a chain of inequalities, which gives:
\[ C_\mathbf{A} \leq t_1 N_{1, \mathbf{A}'} + t_2 N_{2, \mathbf{A}'} + ... + t_M N_{M, \mathbf{A}'} + t_M(N_{1, \mathbf{A}}-N_{1, \mathbf{A}'} + N_{2, \mathbf{A}} - N_{2, \mathbf{A}} + ... + N_{M, \mathbf{A}} - N_{M, \mathbf{A}'})\]
Finally, observing that $\sum_{a =1}^{M} N_{a, \mathbf{A}} = \sum_{a =1}^{M} N_{a, \mathbf{A}'}$, we observe:
\[ C_\mathbf{A} \leq t_1 N_{1, \mathbf{A}'} + t_2 N_{2, \mathbf{A}'} + ... + t_M N_{M, \mathbf{A}'} = C_{\mathbf{A}'}\]
This his establishes that for the $M$-arc case that $C_\mathbf{A} \leq C_{\mathbf{A}'}$ for all \textit{feasible} algorithms $\textbf{A}'$. \qed

\subsection{Proof of Lemma~\ref{lem:greed-opt}} \label{proof-greed-opt}
We first note that Algorithm~\ref{alg:greedy}, denoted as $\mathbf{A}$, certainly satisfies Constraints~\eqref{eq:prob-constraint}-\eqref{eq:capacity-constraints}. This is because Algorithm~\ref{alg:greedy} routes each user on exactly one arc, such that the arc is used below capacity at the time of arrival of each user $i$. For the remainder of this proof we denote arc one as $a_1$ and arc two as $a_2$, where $t_1 < t_2$. Note that if $t_1 = t_2$ then any \textit{feasible} algorithm is optimal since all customers incur the same travel time.

To prove our claim, we proceed by induction on the number of users that are allocated by Algorithm~\ref{alg:greedy} on $a_2$. Note that if Algorithm~\ref{alg:greedy} allocates all users to $a_1$ then it is clearly optimal. Thus, for the base case suppose that Algorithm~\ref{alg:greedy} routes one user on the second arc. We will now show that if Algorithm~\ref{alg:greedy} allocates the $q$'th arriving user on $a_2$ then any other \textit{feasible} algorithm $\mathbf{A}'$ must allocate at least one of the first $q$ users on $a_2$. To prove this claim, we proceed by contradiction. In particular, we assume that there is some \textit{feasible} algorithm $\mathbf{A}'$ that allocates none of the first $q$ users to $a_2$. This implies that the first $q-1$ users must be allocated by algorithm $\mathbf{A}'$ on $a_1$. Note that Algorithm~\ref{alg:greedy} also allocates the first $q-1$ users on $a_1$. Next, since Algorithm~\ref{alg:greedy} allocated the $q$'th arriving user to $a_2$ it must hold that $a_1$ is at capacity at the time user $q$ arrives; otherwise Algorithm~\ref{alg:greedy} would have allocated user $q$ on $a_1$. This establishes our claim that algorithm $\mathbf{A}'$ must allocate at least one of the first $q$ users to $a_2$. Finally, as Algorithm~\ref{alg:greedy} only allocates one user on $a_2$, it follows that $N_{1, \mathbf{A}} \geq N_{1, \mathbf{A}'}$, as any other \textit{feasible} algorithm also allocates at least one user to $a_2$. This establishes the base case.

For the inductive step, we assume that if $k$ users are allocated by Algorithm~\ref{alg:greedy} on $a_2$, then any other feasible algorithm $\mathbf{A}'$ must also allocate at least $k$ users on $a_2$. As established for the base case, we further assume that for each of the $k$ users $q_1, ..., q_k$ allocated on $a_2$ by Algorithm~\ref{alg:greedy}, there is a corresponding user $\Tilde{q}_l \leq q_l$ for each $l \in [k]$ that is allocated to $a_2$ by any other \textit{feasible} algorithm. Now, suppose that Algorithm~\ref{alg:greedy} allocates $k+1$ users on $a_2$. In this case, we first consider the subset of the users in the system up until the $q_k$'th user is routed by Algorithm~\ref{alg:greedy} on $a_2$. For this subset of users, we observe from the inductive assumption that any \textit{feasible} algorithm $\mathbf{A}'$ must route at least $k$ users $\Tilde{q}_l$ for $l \in [k]$ on $a_2$ such that $\Tilde{q}_l \leq q_l$ for each $l \in [k]$.

Given the above observations, we now prove the inductive step by contradiction. In particular, we suppose that there exists a \textit{feasible} algorithm $\mathbf{A}'$ that only allocates $k$ users on $a_2$. This implies that every user after the $\Tilde{q}_k$'th user is allocated on $a_1$ by algorithm $\mathbf{A}'$. Note that this allocation is exactly the same as the allocation of Algorithm~\ref{alg:greedy} for users between but not including $q_k$ and $q_{k+1}$ on~$a_1$. Next, since Algorithm~\ref{alg:greedy} allocated the $q_{k+1}$'th arriving user to $a_2$ it must hold that $a_1$ is at capacity at the time user $q_{k+1}$ arrives. Finally, since $\Tilde{q}_k \leq q_k$ it follows that $\mathbf{A}'$ must allocate at least one of the users between $\Tilde{q}_{k}$ and $q_{k+1}$ on $a_2$ to remain \textit{feasible}, i.e., avoid violating the capacity constraint. Thus, algorithm $\mathbf{A}'$ allocates at least $k+1$ users on $a_2$, which establishes that $N_{1, \mathbf{A}} \geq N_{1, \mathbf{A}'}$ when $k+1$ users are allocated by Algorithm~\ref{alg:greedy} on $a_2$. This establishes the inductive step. \qed

\subsection{Proof of Lemma~\ref{lem:M-greater-than-3}} \label{proof-M-greater-than-3}
To prove this lemma, we consider a three arc network where the cost of arc three is arbitrary and note that the proof can be generalized to any $M$ arc network for $M \geq 3$. We denote the three arcs in the network as $a_1$, $a_2$, and $a_3$ respectively.

We will show that an adversary can construct input sequences such that any \textit{feasible} deterministic algorithm will allocate users on arc $a_3$, while the optimal allocation is to assign users to the first two arcs. Since an adversary can choose the cost of $a_3$ to be arbitrarily large, this establishes that the ratio between the total travel cost of the online solution of any deterministic algorithm and that of the optimal offline solution is unbounded for the $M \geq 3$ arc \acrshort{otr-i} problem.

We first fix a three arc parallel network, where the arc capacities are $c_1 = 1, c_2 = 1, c_3 = 10$, and the arc costs are $t_1 = 5, t_2 = 10.01$, while $t_3$ is the cost of arc three that is strictly greater than $t_2$. Note that in any such three arc network a \textit{feasible} deterministic algorithm $\mathbf{A}$ must allocate the first arriving user on one of the three arcs. We now analyze each of these cases in turn.

\paragraph{Case (i):} Suppose $\mathbf{A}$ allocates the first arriving user to $a_3$. Then an adversary can construct an input sequence such that only one user arrives in the system. Since the optimal allocation is to assign this user on $a_1$, the ratio between the cost of $\mathbf{A}$ and that of the optimal allocation is $\frac{t_3}{t_1}$.

\paragraph{Case (ii):} Suppose $\mathbf{A}$ allocates the first arriving user to $a_2$. Then consider the following input sequence, where user one arrives at time $0$, user two arrives at time $10$, and user three arrives at time $10.001$. Note that an optimal allocation for this input sequence is to allocate users one and two on $a_1$, and user three on $a_2$. That is, the optimal allocation does not allocate a user on $a_3$. However, since $\mathbf{A}$ allocates user one on $a_2$, it must be that either one of users two or three is allocated on $a_3$. To see this, we note that at the time both users two and three arrive, user one is still on $a_2$ since the travel time of $a_2$ is strictly greater than the arrival time of user three. Since the capacity of both $a_1$ and $a_2$ are one, it must be that $\mathbf{A}$ allocates at least one of the last two users on $a_3$. Thus, we have that the ratio of the cost of $\mathbf{A}$ to that of the optimal allocation is at least $\frac{t_1 + t_2 + t_3}{2 t_1 + t_2}$.

\paragraph{Case (iii):} Suppose $\mathbf{A}$ allocates the first arriving user to $a_1$. Then consider the following input sequence, where user one arrives at time $0$, user two arrives at time $0.15$, user three arrives at time $5.2$, and user four arrives at time $10.1$. Note that an optimal allocation for this input sequence is to allocate user one on $a_2$, users two and three on $a_1$, and user four on $a_2$. That is, the optimal allocation does not allocate a user on $a_3$. However, since $\mathbf{A}$ allocates user one on $a_1$, it must be that at least one of the last three users is allocated on $a_3$. To see this, we note that $\mathbf{A}$ cannot allocate user two to $a_1$ since user one occupies $a_1$ between the time period $[0, 5]$ since the cost of $a_1$ is $5$. Thus, $\mathbf{A}$ must either allocate user two to $a_2$ or $a_3$. If $\mathbf{A}$ allocates user two to $a_3$ this will lead to an unbounded ratio of the total travel cost of the online solution to that of the optimal offline solution since an adversary can choose the travel time of $a_3$ to be arbitrarily large. Thus, we consider the case when $\mathbf{A}$ allocates user two to $a_2$. 

Now, we observe that one of the users three or four must be allocated to $a_3$ since both users' arrival time is before the departure of user two from $a_2$. Since the capacity of both $a_1$ and $a_2$ are one, $\mathbf{A}$ must allocate at least one of the last two users on $a_3$. Thus, we have that the ratio of the cost of $\mathbf{A}$ to that of the optimal allocation is at least $\frac{2t_1 + t_2 + t_3}{2 t_1 + 2 t_2}$.

Note that since $t_1 \leq t_2 \leq t_3$, it follows that $\frac{t_3}{t_1} \geq \frac{t_1 + t_2 + t_3}{2 t_1 + t_2} \geq \frac{2t_1 + t_2 + t_3}{2 t_1 + 2 t_2}$. Thus, the ratio of the total travel cost of the online solution of any deterministic algorithm to that of the optimal offline solution is at least $\frac{2t_1 + t_2 + t_3}{2 t_1 + 2 t_2} = 1 + \frac{t_3-t_2}{2 t_1 + 2 t_2}$.

Finally, we note that this ratio is unbounded since an adversary can choose the travel time $t_3$ to be arbitrarily large for fixed values of $t_1$ and $t_2$. \qed


\subsection{Proof of Lemma~\ref{lem:non-identical-OTR}} \label{proof-non-identical-OTR}
To prove this theorem, we restrict our attention to a setting with two user values of time and note that the proof can be generalized for more than two values of time of users. This is because the ratio between the online and the optimal offline solution only depends on the largest and smallest user values of time.

\paragraph{Notation:} We first introduce the notation to define the cost of routing users in the network. Let $N_{a, \mathbf{A}, k}$ and $N_{a, OPT, k}$ be the number of users with a value-of-time $k \in \{ \theta', \theta''\}$ that are routed on arc $a$ by algorithm $\mathbf{A}$ and the optimal offline solution respectively. Then for a two arc parallel network, the total cost for any feasible algorithm that satisfies Constraints~\eqref{eq:prob-constraint}-\eqref{eq:capacity-constraints} is given by:
\begin{align}
    C_{\mathbf{A}} = \theta'' N_{1, \mathbf{A}, \theta''} t_1 + \theta' N_{1, \mathbf{A}, \theta'} t_1 + \theta'' N_{2, \mathbf{A}, \theta''} t_2 + \theta' N_{2, \mathbf{A}, \theta'} t_2
\end{align}
In this setting, the competitive ratio for any algorithm $\mathbf{A}$ on an input sequence $\mathcal{I}$ can be denoted as follows:
\begin{align}
    CR(\mathbf{A}) = \max_{\mathcal{I} \in \Omega} \frac{\theta'' N_{1, \mathbf{A}, \theta''} t_1 + \theta' N_{1, \mathbf{A}, \theta'} t_1 + \theta'' N_{2, \mathbf{A}, \theta''} t_2 + \theta' N_{2, \mathbf{A}, \theta'} t_2}{\theta'' N_{1, OPT, \theta''} t_1 + \theta' N_{1, OPT, \theta'} t_1 + \theta'' N_{2, OPT, \theta''} t_2 + \theta' N_{2, OPT, \theta'} t_2}
\end{align}
Here it must hold that: $ N_{1, OPT, \theta''} + N_{2, OPT, \theta''} = N_{\theta''} = N_{1, \mathbf{A}, \theta''} + N_{2, \mathbf{A}, \theta''}$ and $N_{1, OPT, \theta'} + N_{2, OPT, \theta'} = N_{\theta'} =  N_{1, \mathbf{A}, \theta'} + N_{2, \mathbf{A}, \theta'}$, since the number of users with the same value-of-time for the input sequence $\mathcal{I}$ is fixed.

\paragraph{Lower Bound on Worst-Case $\frac{ALG(\mathcal{I},\mathbf{A})}{OPT(\mathcal{I})}$ Ratio of Deterministic Algorithms:} To show that no deterministic algorithm can perform better than the ratio of $\frac{ALG(\mathcal{I},\mathbf{A})}{OPT(\mathcal{I})} = \frac{\theta'' t_2 + \theta' t_1}{\theta'' t_1 + \theta' t_2}$, we consider an instance of two parallel arcs, each having a capacity of 1, i.e., $c_1 = 1 = c_2$. Suppose an adversary creates an instance where the first arriving user has a value-of-time $\theta'$. Then a deterministic algorithm can either allocate this user to arc one ($a_1$) or arc two ($a_2$).

First, suppose that a deterministic algorithm allocates the user on $a_1$ with a cost $t_1 \leq t_2$. Then an adversary can construct an instance such that the second user has value-of-time $\theta''$ and arrives at time $\epsilon>0$, where $\epsilon < t_1$. For this instance, a deterministic algorithm must allocate the second user on $a_2$ with a higher cost of $t_2$. For this input the cost of a deterministic algorithm is $\theta'' t_2 + \theta' t_1$; however the cost of the optimal allocation would have been $\theta'' t_1 + \theta' t_2$, which would result from allocating the user with value-of-time $\theta''$ on $a_1$ and the user of type $\theta'$ on $a_2$. Thus, for this instance $\mathcal{I}$, $\frac{ALG(\mathcal{I},\mathbf{A})}{OPT(\mathcal{I})} = \frac{\theta'' t_2 + \theta' t_1}{\theta'' t_1 + \theta' t_2}$.

On the other hand, suppose that a deterministic algorithm allocates the first user of type $\theta'$ on $a_2$. Then an adversary can construct an instance of just one user, i.e., no further users arrive into the system. For this instance $\mathcal{I}$, $\frac{ALG(\mathcal{I},\mathbf{A})}{OPT(\mathcal{I})} = \frac{t_2}{t_1} \geq \frac{\theta'' t_2 + \theta' t_1}{\theta'' t_1 + \theta' t_2}$ since the optimal algorithm would have allocated the user on $a_1$. Thus no deterministic algorithm can obtain a lower ratio of $\frac{ALG(\mathcal{I},\mathbf{A})}{OPT(\mathcal{I})}$ than $\frac{\theta'' t_2 + \theta' t_1}{\theta'' t_1 + \theta' t_2}$.

\paragraph{Greedy Achieves the Lower Bound on the $\frac{ALG(\mathcal{I},\mathbf{A})}{OPT(\mathcal{I})}$ Ratio for Deterministic Algorithms:} Now, to see that the greedy algorithm $\mathbf{A}$ achieves this competitive ratio for any network parameter set $\mathcal{S}$, we will prove that for any instance $\mathcal{I}$ that $\frac{ALG(\mathcal{I}, \mathbf{A})}{OPT(\mathcal{I})} \leq \frac{\theta'' t_2 + \theta' t_1}{\theta'' t_1 + \theta' t_2}$.

To prove this, we start from the definition of the competitive ratio of the greedy allocation, $\mathbf{A}$, to the optimal allocation, $OPT$:

\[ \frac{ALG(\mathcal{I}, \mathbf{A})}{OPT(\mathcal{I})} = \frac{\theta'' N_{1, \mathbf{A}, \theta''} t_1 + \theta' N_{1, \mathbf{A}, \theta'} t_1 + \theta'' N_{2, \mathbf{A}, \theta''} t_2 + \theta' N_{2, \mathbf{A}, \theta'} t_2}{\theta'' N_{1, OPT, \theta''} t_1 + \theta' N_{1, OPT, \theta'} t_1 + \theta'' N_{2, OPT, \theta''} t_2 + \theta' N_{2, OPT, \theta'} t_2} \]

Note that since the number of users belonging to the two values of time are fixed at $N_{\theta'}$ and $N_{\theta''}$ the above expression can be written as:

\begin{align} \label{eq:alg-opt-frac}
    \frac{ALG(\mathcal{I}, \mathbf{A})}{OPT(\mathcal{I})} = \frac{t_1 [\theta'' N_{1, \mathbf{A}, \theta''} t_1 + \theta' N_{1, \mathbf{A}, \theta'}] + t_2 [\theta'' (N_{\theta''} - N_{1, \mathbf{A}, \theta''}) + \theta' (N_{\theta'} - N_{1, \mathbf{A}, \theta'})]}{t_1 [\theta'' N_{1, OPT, \theta''} + \theta' N_{1, OPT, \theta'}] + t_2[\theta'' (N_{\theta''} - N_{1, OPT, \theta''}) + \theta' (N_{\theta'} - N_{1, OPT, \theta'}) ]}
\end{align}
The above equation holds since $N_{1, \mathbf{A}, k} + N_{2, \mathbf{A}, k} = N_k$ and $N_{1, OPT, k} + N_{2, OPT, k} = N_k$ for $k \in \{\theta', \theta'' \}$.

We now seek for the maximum possible ratio $\frac{ALG(\mathcal{I}, \mathbf{A})}{OPT(\mathcal{I})}$ for any instance, subject to the constraint that $$N_{1, \mathbf{A}, \theta''} + N_{1, \mathbf{A}, \theta'} \geq N_{1, OPT, \theta''} + N_{1, OPT, \theta''},$$ as the greedy algorithm will allocate at least as many users as $OPT$ does on the first arc, i.e., the arc with a lower cost as $t_1 \leq t_2$. This follows from the result of Lemma~\ref{lem:greed-opt}. 

To find an upper bound on the maximum ratio of $\frac{ALG(\mathcal{I}, \mathbf{A})}{OPT(\mathcal{I})}$, we first note that the inequality 
$$N_{1, \mathbf{A}, \theta''} + N_{1, \mathbf{A}, \theta'} \geq N_{1, OPT, \theta''} + N_{1, OPT, \theta''}$$ must be met with equality. This is because if $$N_{1, \mathbf{A}, \theta''} + N_{1, \mathbf{A}, \theta'} > N_{1, OPT, \theta''} + N_{1, OPT, \theta''},$$ then the numerator could be strictly increased or the denominator could be strictly decreased whilst maintaining $$N_{1, \mathbf{A}, \theta''} + N_{1, \mathbf{A}, \theta'} \geq N_{1, OPT, \theta''} + N_{1, OPT, \theta''}.$$ This also implies that the number of users allocated to arc two are the same.

From the above discussion, we have that the number of users allocated on arcs $a_1$ and $a_2$ are equal for both the greedy algorithm as well as $OPT$ to maximize the ratio of $\frac{ALG(\mathcal{I}, \mathbf{A})}{OPT(\mathcal{I})}$. Thus, the maximum ratio of $\frac{ALG(\mathcal{I}, \mathbf{A})}{OPT(\mathcal{I})}$ is achieved when the high value-of-time users are allocated on $a_2$ by the greedy algorithm while the same users are allocated on $a_1$ by the optimal offline solution.  Note that if even one high value-of-time user was allocated on $a_1$ then the bound for the competitive ratio would only be lower. This is because $ALG(\mathcal{I}, \mathbf{A})$ would be smaller. Next, for each high value-of-time user we must have at least one low value-of-time user that is allocated by the greedy algorithm on $a_1$. Since the number of users on $a_1$ and $a_2$ are the same, we have that:
\begin{align}
    \frac{ALG(\mathcal{I}, \mathbf{A})}{OPT(\mathcal{I})} &\leq \frac{\theta' (N_{\theta''} + N_{1, OPT, \theta'}) t_1 + \theta'' N_{\theta''} t_2 + \theta' (N_{\theta'} - N_{1, \mathbf{A}, \theta'}) t_2}{\theta'' N_{\theta''} t_1 + \theta' N_{1, OPT, \theta'} t_1 + \theta' (N_{\theta''} + N_{\theta'} - N_{1, \mathbf{A}, \theta'}) t_2}
\end{align}

The above ratio is identical to: $\frac{a+b}{c+b}$, where $b = \theta' (N_{\theta'} - N_{1, \mathbf{A}, \theta'}) t_2$, and $a \geq c$ as $\frac{ALG(\mathcal{I}, \mathbf{A})}{OPT(\mathcal{I})} \geq 1$. Taking the derivative of $\frac{a+b}{c+b}$ with respect to $b$, we get $\frac{c-a}{(c+b)^2} \leq 0$. Due to the sign of the derivative, the maximum is attained when $b = 0$, i.e., $N_{\theta'} = N_{1, \mathbf{A}, \theta'}$, which implies that:

\[ \frac{ALG(\mathcal{I}, \mathbf{A})}{OPT(\mathcal{I})} \leq \frac{\theta' N_{\theta'} t_1 + \theta'' N_{\theta''} t_2}{\theta'' N_{\theta''} t_1 + \theta' N_{1, OPT, \theta'} t_1 + \theta' N_{\theta''} t_2} \]

Now, since we have $N_{\theta'} = N_{\theta''} + N_{1, OPT, \theta'}$, we get that:
\[ \frac{ALG(\mathcal{I}, \mathbf{A})}{OPT(\mathcal{I})} \leq \frac{\theta' N_{\theta'} t_1 + \theta'' N_{\theta''} t_2}{\theta'' N_{\theta'} t_1 + \theta' N_{\theta''} t_2} \]
Taking the derivative of the above equation with respect to $N_{\theta'}$, we get that the derivative is negative, and, since $N_{\theta'} = N_{\theta''} + N_{1, OPT, \theta'}$, the maximum ratio is achieved when $N_{\theta'} = N_{\theta''}$. Thus, we get that:
\[ \frac{ALG(\mathcal{I}, \mathbf{A})}{OPT(\mathcal{I})} \leq \frac{\theta' N_{\theta''} t_1 + \theta'' N_{\theta''} t_2}{\theta'' N_{\theta''} t_1 + \theta' N_{\theta''} t_2} \]
This gives us that:

\[ \frac{ALG(\mathcal{I}, \mathbf{A})}{OPT(\mathcal{I})} \leq \frac{\theta'' t_2 + \theta' t_1}{\theta'' t_1 + \theta' t_2} \]
which proves our claim that for any instance $\mathcal{I}$ that $\frac{ALG(\mathcal{I}, \mathbf{A})}{OPT(\mathcal{I})} \leq \frac{\theta'' t_2 + \theta' t_1}{\theta'' t_1 + \theta' t_2}$. This establishes that the competitive ratio of greedy is $\frac{\theta'' t_2 + \theta' t_1}{\theta'' t_1 + \theta' t_2}$, since the competitive ratio of no deterministic algorithm can be lower than this. As a result we have established that the greedy algorithm as presented in Algorithm~\ref{alg:greedy} is the optimal deterministic algorithm, i.e., it achieves the best possible competitive ratio of $\frac{\theta'' t_2 + \theta' t_1}{\theta'' t_1 + \theta' t_2}$ amongst all deterministic algorithms. \qed

\subsection{Proof of Lemma~\ref{lem:general-road-networks}} \label{proof-general-road-networks}
To prove the lemma, we consider a series parallel road network as depicted in Figure~\ref{series-parallel-network-simple}, where the cost $t_4$ of arc $L_4$ can be arbitrary and the capacities of $L_2$ and $L_4$ can be large, e.g., 10 users can concurrently use each of $L_2$ and $L_4$. We consider a deterministic algorithm that allocates the first arriving user on arcs (i) $L_1, L_4$ or $L_2, L_4$, i.e., a route containing arc $L_4$, (ii) $L_1, L_3$, and (iii) $L_2, L_3$. We show that in each of these cases, there is an input sequence for which a deterministic algorithm will allocate at least one user on arc $L_4$ while the optimal offline allocation will not allocate users on arc $L_4$. Since the cost of arc $L_4$ can be made arbitrarily large, the ratio of the cost of any deterministic algorithm to the optimal offline solution is unbounded.

\begin{figure}[!h]
      \centering
      \includegraphics[width=0.7\linewidth]{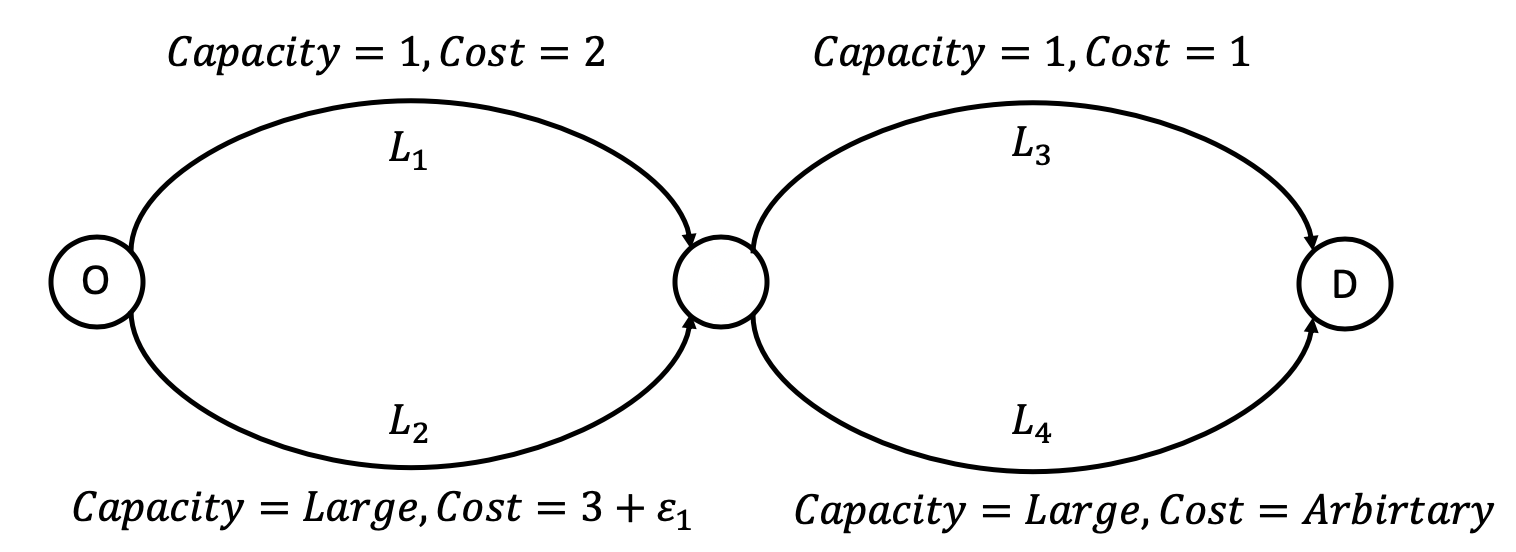}
      \caption{An example of a series parallel network, where the cost of arc $L_4$ can be chosen by an adversary. Here ``O'' and ``D'' denote the origin and destination nodes in this graph.}
      \label{series-parallel-network-simple}
   \end{figure}

\paragraph{Case (i):} First, suppose that a deterministic algorithm allocates the user arriving at time 0 on a route containing arc $L_4$; then, an adversary can consider an input of just one user, i.e., no further users enter the system. For this one user, the optimal offline algorithm would have allocated the user on the arcs $L_1, L_3$. Since the user is allocated on a route consisting of arc $L_4$, the ratio of the cost between such a deterministic algorithm and the optimal offline solution is unbounded.

\paragraph{Case (ii):} Next, suppose that a deterministic algorithm allocates the user arriving at time 0 on arcs $L_1, L_3$. Then an adversary can construct an input sequence where user one arrives at time 0, user two arrives at time $\epsilon$ and user three arrives at time $1+\epsilon'$, where $\epsilon' < \epsilon < \epsilon_1 \leq 0.1$. Note that $3+\epsilon_1$ is the cost of arc $L_2$. With this input a deterministic algorithm, having to satisfy the feasibility constraints~\eqref{eq:prob-constraint-general-road}-\eqref{eq:capacity-constraints-general-road}, will have to allocate one of users two or three onto $L_4$. On the other hand, the optimal offline solution would allocate the first user to arcs $L_2, L_3$, user two on arcs $L_1, L_3$ and user three on arcs $L_2, L_3$, and thus never use arc $L_4$. Thus, the ratio of the cost between such a deterministic algorithm and the optimal offline solution is unbounded.

\paragraph{Case (iii):} Finally, suppose that a deterministic algorithm allocates the user arriving at time 0 on arcs $L_2, L_3$. An adversary can then construct an input sequence wherein user two arrives at time $1-\epsilon''$ for $\epsilon'' \leq 0.1$. If a deterministic algorithm allocates this user on $L_2$, she must be routed on $L_2, L_4$ to ensure feasibility because $L_3$ has a capacity of one. On the other hand, if the second user is allocated on a route containing $L_1$, she must be allocated to route $L_1, L_4$. This is because user one uses $L_3$ for the time period $[3+\epsilon_1, 4+\epsilon_1]$, which coincides with the time at which user two exits $L_1$. Thus, the second user cannot use $L_1, L_3$ and will be allocated to $L_1, L_4$. On the other hand, the optimal offline solution would allocate user one on $L_1, L_3$ and user two on $L_2, L_3$, thereby never using $L_4$. Thus, the ratio of the cost between such a deterministic algorithm and the optimal offline solution is unbounded. \qed

\subsection{Proof of Lemma~\ref{lem:general-congestion-functions}} \label{proof-general-congestion-functions}
To prove this lemma, we consider a two arc network with arcs $a_1$, $a_2$ that have capacities $c_1 = 2$ and $c_2 = 1$. Furthermore, we let the cost of arc two be a constant $f_2(x) = 10.01$, and so $f_2(1) = 10.01$. Next, we let the cost of arc one be defined by a function such that $f_1(1) = 5$ and $f_1(2) = \Bar{d}$, where $\Bar{d}$ can be chosen by an adversary. Note that a linear function $f_1$ that passes through the points $(1, 5)$ and $(2, \Bar{d})$ will suffice.

We will show that an adversary can construct input sequences such that any \textit{feasible} deterministic algorithm $\mathbf{A}$ will allocate two users simultaneously on arc $a_1$ to incur a cost $\Bar{d}$, while the optimal allocation is to assign users such that no two users concurrently use arc one. As a result, the ratio of the cost between any deterministic algorithm and the optimal offline solution is unbounded.

Since in any such two arc network a \textit{feasible} deterministic algorithm $\mathbf{A}$ must allocate the first arriving user on one of the two arcs, we consider two cases.

\paragraph{Case (i):} Suppose $\mathbf{A}$ allocates the first arriving user to $a_2$. Then consider the following sequence, where user one arrives at time $0$, user two arrives at time $10$, and user three arrives at time $10.001$. Note that an optimal allocation for this sequence is to allocate user one on $a_1$, user two on $a_1$, and user three on $a_2$. Since user one departs from the system before user two arrives, the optimal allocation does not allocate two users concurrently on $a_1$. However, since $\mathbf{A}$ allocates user one on $a_2$, it must hold that both users two and three are allocated on $a_1$. To see this, we note that at the times users two and three arrive, user one is still on $a_2$ since the travel time of $a_2$ is strictly greater than the arrival time of user three. Since the capacity of $a_2$ is one, it must hold that $\mathbf{A}$ allocates both users two and three on $a_1$. Since $\Bar{d}$ can be arbitrarily large, we have that the ratio of the cost of such a deterministic algorithm to that of the optimal offline solution is unbounded.

\paragraph{Case (ii):} Suppose $\mathbf{A}$ allocates the first arriving user to $a_1$. Then consider the following input sequence, where user one arrives at time $0$, user two arrives at time $0.015$, user three arrives at time $5.02$, and user four arrives at time $10.01$. Note that an optimal allocation for this input sequence is to allocate user one on $a_2$, users two and three on $a_1$, and user four on $a_2$. Again, as with case (i), the optimal allocation does not allocate two users concurrently on $a_1$. However, since $\mathbf{A}$ allocates user one on $a_1$, it must be that at least two of the last three users are concurrently allocated on $a_1$. To see this, we note that if $\mathbf{A}$ allocates the second arriving user on $a_1$, then we have two users concurrently occupying $a_1$ thereby incurring an arbitrarily large cost. Thus, we consider the case when $\mathbf{A}$ allocates user two to $a_2$. Now, we observe that both users three and four must be allocated to $a_1$ since both users' arrival time is before the departure of user two from $a_2$. Since the capacity of $a_2$ is one, it must be that $\mathbf{A}$ allocates both of the last two users on $a_1$. Thus, again we have that the ratio between the total travel cost of the online solution of $\mathbf{A}$ to that of the optimal allocation is unbounded since $\Bar{d}$ can be arbitrarily large. \qed

\section{Generalizations of \acrshort{otr} Problem}

\subsection{Extending \acrshort{otr} Problem to General (Non-Parallel) Road Networks} \label{framework-general-road}
In this section, we present a generalization of the \acrshort{otr} problem for parallel networks to the setting of general road networks. To model a general road network, we consider a directed graph $G = (V, E)$, where $V$ is its set of vertices and $E$ is its set of edges. Travelling through arc $a \in E$ requires $t_a$ time units, and each arc $a$ has a capacity $c_a$. We collect these parameters in $\mathcal{S} := \{ \{ c_a \}_{a \in E}, \{ t_a \}_{a \in E} \}$. We denote an input sequence of $n$ users by $\mathcal{I} := \{ \{ \theta_i \}_{i \in [n]}, \{ \tau_i \}_{i \in [n]}, \{ w_i \}_{i \in [n]} \}$, stating for each user $i$ its time of appearance $\tau_i\in\mathbb{R}_{\ge0}$, its value-of-time $\theta_i \in \Theta$, where $\Theta$ is finite, and its origin destination pair $w_i \in W$. Each origin destination pair $w \in W$ has $R_w$ routes and each route $r \in R_w$ is specified by an ordered sequence of $l$ edges $\{e_1^r, e_2^r, ..., e_l^r \}$. Here edge $e_1^r$ is a directed edge starting at the origin node of $w$ and edge $e_l^r$ is a directed edge ending at the destination node of $w$. With this notation, the \acrlong{otr} (\acrshort{otr}) problem for general road networks is as follows.

\begin{problem}\label{prob:otr-general-networks} (\acrlong{otr} (\acrshort{otr}) for General Road Networks) Let the parameter set $\mathcal{S}$ and an input sequence $\mathcal{I}$ define a problem instance. Users arrive sequentially in the system and must be irrevocably assigned to a route $r \in R_{w_i}$ at their time of appearance $\tau_i$. The number of users concurrently traversing an arc is limited to the capacity $c_a$ of the arc, and each user traverses an arc $a$ for $t_a$ time units. Then, the \acrshort{otr} problem for general road networks is to construct an assignment of users to routes such that the total cost or time that users spend to traverse their routes is minimal.
\end{problem} 
\noindent If all users are identical, i.e., $\theta_i = \theta_{i'}$ for each $i, i' \in [n]$, we refer to this problem variant as \acrshort{otr-i} for general road networks. %

We evaluate the efficacy of an algorithm $\mathbf{A}$ that solves Problem~\ref{prob:otr-general-networks} via its competitive ratio 
$$CR(\mathbf{A}) = \max_{\mathcal{I} \in \Omega} \frac{ALG(\mathcal{I}, \mathbf{A})}{OPT(\mathcal{I})},$$
and say that an algorithm is $\alpha$-competitive if $CR(\mathbf{A}) \leq \alpha$. Here $OPT(\mathcal{I})$ denotes the optimal offline solution to the \acrshort{otr} problem for general road networks and $ALG(\mathcal{I}, \mathbf{A})$ denotes the online solution achieved by algorithm $\mathbf{A}$ on input sequence $\mathcal{I}$. The input instances $\mathcal{I}$ belong to a set $\Omega$, which corresponds to user arrivals such that $\theta_i \geq 0$ for all users $i$ and $\tau_i \geq 0$, with the first user arriving at time $0$. We further assume without loss of generality that $\tau_1 = 0 < \tau_2 < ... < \tau_n $, i.e., no two users arrive in the system at the same point in time.

We compute $OPT(\mathcal{I})$ by solving the following optimization problem.
\begin{mini!}|s|[2]<b>
	{\substack{x_{ir}, \\ \forall i \in [n], r \in \{R_w\}_{w \in W}}}{\sum_{i = 1}^{n} \sum_{a \in E} x_{ir} \theta_i t_a \delta_{a, r}\label{eq:OPT-integer-Obj-general-road}}
	{\label{eq:Example22222}}
	{}
	\addConstraint{\sum_{r = 1}^{|R_w|} x_{ir}}{=1 \quad \forall i \in [n] \label{eq:prob-constraint-general-road}}
	\addConstraint{x_{ir}}{\in \{0, 1\}, \quad \forall r \in [|R_{w_i}|, \forall i \in [n] ] \label{eq:binary-constraint-general-road}}
	\addConstraint{\sum_{i=1}^n \sum_{r:a \in r} \mathbbm{1}_{\tau_{k} \in [\tau_{i} + \sum_{a' \in r:a' < a} t_{a'}, \tau_{i} + \sum_{a' \in r:a' \leq a} t_{a'}]} x_{ir}}{\leq c_a, \quad \forall a \in E, \forall \tau_k \in \{\tau_1, ..., \tau_n \} \label{eq:capacity-constraints-general-road}}
\end{mini!}

Here we introduce binary variables $x_{ir}$ to denote whether user $i$ is assigned to route $r \in R_{w_i}$ ($x_{ir} = 1$) or not ($x_{ir}=0$). Further, we assume the constraints of this problem are feasible where, \eqref{eq:prob-constraint-general-road} are assignment constraints,~\eqref{eq:binary-constraint-general-road} are binary constraints and~\eqref{eq:capacity-constraints-general-road} are capacity constraints, and $\delta_{a, r}$ is the indicator function stating whether arc $a$ belongs to route $r$ ($\delta_{a, r}=1$) or not ($\delta_{a, r}=0$). In the capacity constraints, $\sum_{a' \in r:a' < a} t_{a'}$ denotes the sum of the travel times on edges preceding edge $a$ on route $r$, and $\sum_{a' \in r:a' \leq a} t_{a'}$ is the sum of the travel times on edges preceding as well as including edge $a$ on route $r$.

\subsection{Extending OTR Problem to General Cost Functions} \label{framework-general-congestion}
In this section, we present a generalization of the \acrshort{otr} problem to the setting when the travel time of the arcs is modelled by a cost function. As with the \acrshort{otr} problem, we consider a graph with $M \geq 2$ parallel arcs and two vertices representing an origin and a destination. Travelling through arc $a$ requires $f_a(y_a)$ time units, when $y_a$ users use arc $a$, and each arc $a$ has a capacity $c_a$. Here $f_a:\mathbb{R}_{\geq 0} \mapsto \mathbb{R}_{\geq 0}$ is non-decreasing in its argument. We collect these parameters in $\mathcal{S} := \{ \{ c_a \}_{a \in [M]}, \{ f_a \}_{a \in [M]} \}$. We denote an input sequence of $n$ users by $\mathcal{I} := \{ \{ \theta_i \}_{i \in [n]}, \{ \tau_i \}_{i \in [n]} \}$, stating for each user $i$ its time of appearance $\tau_i\in\mathbb{R}_{\ge0}$ and its value-of-time $\theta_i \in \Theta$, where $\Theta$ is finite. With this notation, the \acrshort{otr} problem with general cost functions is as follows.

\begin{problem}\label{prob:otr-general-congestion} (Online Traffic Routing (\acrshort{otr}) with General Cost Functions) Let the parameter set $\mathcal{S}$ and an input sequence $\mathcal{I}$ define a problem instance. Users arrive sequentially in the system and must irrevocably be assigned to an arc at their time of appearance $\tau_i$. The number of users concurrently traversing an arc is limited to the capacity $c_a$ of the arc, and each user traverses an arc $a$ for $f_a(y_a+1)$ time units, where $y_a$ users are using arc $a$ at time $\tau_i$. Then, the \acrshort{otr} problem for general cost functions is to construct an assignment of users to arcs such that the total cost or time that users spend to traverse an arc is minimal.
\end{problem} 
\noindent If all users are identical, i.e., $\theta_i = \theta_{i'}\ \forall\ i, i' \in [n]$, we refer to this problem variant as \acrshort{otr-i} with general cost functions. %

We evaluate the efficacy of an algorithm $\mathbf{A}$ that solves Problem~\ref{prob:otr-general-congestion} via its competitive ratio 
$$CR(\mathbf{A}) = \max_{\mathcal{I} \in \Omega} \frac{ALG(\mathcal{I}, \mathbf{A})}{OPT(\mathcal{I})},$$
and say that an algorithm is $\alpha$-competitive if $CR(\mathbf{A}) \leq \alpha$. Here $OPT(\mathcal{I})$ denotes the optimal offline solution to the \acrshort{otr} problem for general cost functions and $ALG(\mathcal{I}, \mathbf{A})$ denotes the online solution achieved by algorithm $\mathbf{A}$ on input sequence $\mathcal{I}$. The input instances $\mathcal{I}$ belong to a set $\Omega$, which corresponds to user arrivals such that $\theta_i \geq 0$ for all users $i$, and $\tau_i \geq 0$, with the first user arriving at time $0$. We further assume without loss of generality that $\tau_1 = 0 < \tau_2 < ... < \tau_n $, i.e., no two users arrive in the system at the same point in time, and that there exists a feasible assignment. 

We compute $OPT(\mathcal{I})$ by solving the following optimization problem.
\begin{mini!}|s|[2]<b>
	{\substack{x_{ia}, \\ \forall i \in [n], a \in [M]}}{\sum_{i = 1}^n \sum_{a = 1}^M x_{ij} \theta_i f_a \left(x_{ia} + y_a(\tau_i) \right) \label{eq:OPT-integer-Obj-general-congestion}}
	{\label{eq:Example32233}}
	{OPT(\mathcal{I})=}
	\addConstraint{\sum_{a = 1}^{M} x_{ia}}{=1 \quad \forall i \in [n] \label{eq:prob-constraint-general-congestion}}
	\addConstraint{x_{ia}}{\in \{0, 1\}, \quad \forall i\in [n], a \in [M] \label{eq:binary-constraint-general-congestion}}
	\addConstraint{\sum_{i=1}^n \mathbbm{1}_{\tau_{k} \in [\tau_{i}, \tau_{i} + f_j \left(x_{ia} + y_a(\tau_i) \right)]} x_{ia}}{\leq c_a, \quad \forall a \in [M], \forall \tau_k \in \{\tau_1, ..., \tau_n \} \label{eq:capacity-constraints-general-congestion}}
	\addConstraint{y_a(\tau_i)}{= \sum_{k=1}^n \mathbbm{1}_{\tau_{k} < \tau_i, \tau_k+f_a(x_{ka}+y_a(\tau_{k}))>\tau_i } x_{ia}, \quad \forall a \in [M], \forall \tau_k \in \{\tau_1, ..., \tau_n \} \label{eq:recursive-eqn}}
\end{mini!}

Here, we introduce binary variables $x_{ia}$ to denote whether user $i$ is assigned to arc $a$ ($x_{ia} = 1$) or not ($x_{ia}=0$). Further, \eqref{eq:prob-constraint-general-congestion} are assignment constraints,~\eqref{eq:binary-constraint-general-congestion} are binary constraints and~\eqref{eq:capacity-constraints-general-congestion} are capacity constraints. Equations~\eqref{eq:recursive-eqn} represent the total number of users on each arc $a$ at each time $\tau_i$ a user arrives.


\printglossary[type=\acronymtype]

\end{document}